\documentclass[a4paper,10pt]{amsart}

\usepackage{comment}
\usepackage[font+=Large, labelformat=empty, skip=10pt]{subcaption}

\usepackage{amssymb,amsmath,amsthm,ascmac,array,bm}
\usepackage{xcolor,graphicx}

\usepackage{mathtools}

\usepackage{enumitem}

\usepackage[
colorlinks=true,bookmarks=true,
bookmarksnumbered=true,bookmarkstype=toc,linktocpage=true
]{hyperref}

\newtheorem{thm}{Theorem}[section]

\newtheorem{lem}[thm]{Lemma}

\newtheorem{rem}[thm]{Remark}
\newtheorem{assump}[thm]{Assumption}

\numberwithin{equation}{section}
\allowdisplaybreaks

\newcommand{\mcl}{\mathcal{L}}

\newcommand{\mbbn}{\mathbb{N}}

\newcommand{\mbbr}{\mathbb{R}}

\newcommand{\al}{\alpha} \newcommand{\be}{\beta} 
\newcommand{\ep}{\epsilon} \newcommand{\vp}{\varphi} 
\newcommand{\del}{\delta} \newcommand{\D}{\Delta} 
\newcommand{\sig}{\sigma}  
 \newcommand{\Gam}{\Gamma}
\newcommand{\lam}{\lambda} 
\newcommand{\p}{\partial}
\newcommand{\ra}{\rangle} \newcommand{\la}{\langle}
\newcommand{\cil}{\xrightarrow{\mcl}} 
\newcommand{\cip}{\xrightarrow{p}} 

\newcommand{\argmin}{\mathop{\rm argmin}}
\newcommand{\argmax}{\mathop{\rm argmax}}

\def\nn{\nonumber}

\def\tcr#1{\textcolor{black}{#1}}

\def\rev#1{\textcolor{black}{#1}} 

\newcommand{\pr}{P} \newcommand{\E}{E}

\newcommand{\sgn}{{\rm sgn}}

\def\sumj{\sum_{j=1}^{n}}

\newcommand{\az}{\alpha_0}

\newcommand{\sz}{\sigma_0}
\newcommand{\tz}{\theta_{0}}
\newcommand{\mz}{\mu_{0}}
\newcommand{\tes}{\hat{\theta}_{n}}
\newcommand{\aes}{\hat{\alpha}_{n}}
\newcommand{\bes}{\hat{\beta}_{n}}

\newcommand{\mes}{\hat{\mu}_{n}}
\newcommand{\ses}{\hat{\sigma}_{n}}

\title[On estimation of skewed stable linear regression]
{On estimation of skewed stable linear regression}

\author{Eitaro Kawamo}
\address{Graduate School of Mathematics, Kyushu University, 744 Motooka Nishi-Ku Fukuoka 819-0395, Japan}
\email{kawamo.eitaro.266@s.kyushu-u.ac.jp}

\author{Hiroki Masuda}
\address{Graduate School of Mathematical Sciences, University of Tokyo, 3-8-1 Komaba Meguro-ku Tokyo 153-8914, Japan}
\email{hmasuda@ms.u-tokyo.ac.jp}
\date{\today}


\begin{document}
\maketitle
\mathtoolsset{showonlyrefs=true}

\begin{abstract}
We study the parameter estimation method for linear regression models with possibly skewed stable distributed errors. Our estimation procedure consists of two stages: first, for the regression coefficients, the Cauchy quasi-maximum likelihood estimator (CQMLE) is considered after taking the differences to remove the skewness of noise, and we prove its asymptotic normality and tail-probability estimate; second, as for stable-distribution parameters, we consider the moment estimators based on the symmetrized and centered residuals and prove their $\sqrt{n}$-consistency. To derive the $\sqrt{n}$-consistency, we essentially used the tail-probability estimate of the CQMLE. The proposed estimation procedure has a very low computational load and is much less time-consuming compared with the maximum-likelihood estimator. Further, our estimator can be effectively used as an initial value for the numerical optimization of the log-likelihood.
\end{abstract}

\section{Introduction}

\subsection{Objective and background}
\label{hm:sec_st.reg.setup}

Let $S^{0}_{\alpha} (\beta,\sigma,\delta)$ be the stable distribution with the continuous parameterization, which is given by the characteristic function
\begin{equation}
\varphi (u;\vartheta) := 
\exp \left\{ i \delta u - (\sigma |u|)^{\alpha} \left( 1 + i \beta \sgn(u) \tan \dfrac{\alpha \pi}{2} \left( |\sigma u|^{1 - \alpha} - 1 \right) \right) \right\},
\label{hm:0stable.cf}
\end{equation}
where $\vartheta:=(\al,\be,\sig,\delta)\in(0,2]\times[-1,1]\times(0,\infty)\times\mbbr$.
This is called Nolan's $0$-parametrization \cite[Definition 1.4]{Nol98}, which is continuous in all the parameters.
Note that the possibly skewed Cauchy case is then defined to be the limiting case for $\al\to 1$ through $\lim_{\alpha \to 1} \tan(\al\pi/2) (|u|^{1 - \alpha} - 1) = (2/\pi)\log |u|$:
\begin{equation}
\varphi (u;1,\beta,\sig,\delta) = \lim_{\al\to 1}\varphi (u;\vartheta)
=\exp \left\{ i \delta u - \sigma |u| \left( 1 + i \beta \dfrac{2}{\pi} \sgn (u) \log (\sigma |u|) \right) \right\}.
\label{hm:0stable.cf_al=1}
\end{equation}
It is known that $S_{\alpha}^0 (\beta,\sigma,\delta)$ for $\beta\in(-1,1)$ admits an everywhere positive smooth Lebesgue density; 
we refer to \cite{Sat99} and \cite{Sha69} for details.
We denote by $\phi_{\al,\be}(x)$ the probability density function of $S^{0}_{\al}(\be, 1, 0)$:
\begin{equation}
    \phi_{\al,\be}(x) := \frac{1}{2\pi} \int e^{-ixu}\vp(u;\al,\beta,1,0)du.
\end{equation}
We consider the linear regression model
\begin{equation}
Y_j = X_j \cdot \mu + \sig \ep_j,
\label{hm:model}
\end{equation}
where $X_j\in\mbbr^q$ is a non-random vector of dimension $q$, and
\begin{equation}
\label{hm:noise}
\ep_1,\ep_2,\dots \sim \text{i.i.d.}~S_\al^0(\beta,1,0).
\end{equation}
Here and in what follows the dot denotes the usual inner product.
\tcr{We are concerned with estimation of the unknown parameter $\theta := (\al,\beta,\sig,\mu) \in \Theta$ from a sample $\{(X_j,Y_j)\}_{j=1}^{n}$. Throughout this paper, we use $\Theta$ to denote the parameter space for the full parameter vector $\theta = (\alpha, \beta, \sigma, \mu)$, which is assumed to be a bounded convex subset of 
\begin{equation}
(0,2) \times (-1,1) \times (0,\infty) \times \mathbb{R}^q,
\end{equation}
and also it is assumed that there exists a true value $\tz \coloneqq (\al_0,\beta_0,\sig_0,\mu_0)\in\Theta$.}
We denote by $\pr_\theta$ the distribution of $\{(Y_j,\ep_j)\}_{j=1}^{\infty}$ and by $\E_\theta$ the corresponding expectation operator. 

\rev{The stable (regression) model has various applications such as electric-price estimation through the application of stable alpha regressions \cite{Agu.et.al19} and animal movement GPS telemetry data \cite{Mut17}}.
Parameter estimation of the stable regression model \eqref{hm:model} has a long history.
Typically, one employs the maximum-likelihood estimator (MLE) with numerical integration (for the density) and optimization, see \cite[Chapter 5]{Nol20} and the references therein and also \cite[Section 5.2]{NikSha95} especially for the symmetric case; see also \cite{Mat21} for the i.i.d.-$S_\al^0(\beta,\sig,\delta)$ setting and \cite{AndCalDav09} for the time series model.
However, it can be rather time-consuming and, more seriously, suffers from the local solution problem,
the latter being often inevitable for non-concave log-likelihood cases.
\rev{
Further, in the presence of skewness, most well-known estimators, including the LAD estimator \cite{Kni98}, the trimmed LSE \cite{Nol13}
\footnote{We appreciate the anonymous referee for drawing our attention to Nolan and Ojeda-Revah \cite{Nol13}.} (also \cite{Nol20}), fractional-order moment estimators \cite{NikSha95}, and the empirical-characteristic-function based estimation which goes back to \cite{Pre72} and \cite{Kou80-2} do not effectively work to estimate all the components of $\theta$. 
}

In this paper, we address this problem by introducing simple and easy-to-compute estimators based on the Cauchy density and the method of lower-order fractional moments (see also Remark \ref{hm:rem_LADE}).
\tcr{Our proposed estimators are based on sample residuals of the form $\hat{\ep}_j = Y_j - X_j \cdot \mes$, where $\mes$ is a preliminary estimator of the regression coefficient. These residuals are used to approximate the distributional properties of the noise $\ep_j$ and to construct moment-based estimators for the parameters $(\alpha, \beta, \sigma)$.}
As a result, we derived the estimator that is $\sqrt{n}$-consistent and that can be used as the initial value for numerical optimization of the log-likelihood function; although we will not prove the asymptotic normality of the proposed estimator, it would be trivial from the proof while the resulting asymptotic covariance matrix takes a rather complicated form.
We will refer to the estimators introduced and studied in Sections \ref{sec:CQMLE} and \ref{sec:stable} as \textit{initial estimator}. The initial estimator consists of the Cauchy quasi-maximum likelihood estimator (CQMLE) $\hat{\mu}_n$ and $(\aes,\bes,\ses)$ obtained from the method of moments, introduced in Sections \ref{sec:CQMLE} and \ref{sec:stable}, respectively.

\medskip

\rev{Here is a small numerical example of the MLE.}
In Table \ref{tab:numerical}, we computed the initial estimator and the MLE when $q=3$ with the explanatory variables $X_j$ being generated independently from the uniform distribution $U(1,5)$. The MLE is computed by changing the initial value of the optimization, as indicated in the caption of Table \ref{tab:numerical}. \rev{We used R package \texttt{stabledist} to calculate the value of the probability density function, and Table \ref{tab:numerical} shows the results obtained from a one-time simulation.} 
We see the numerical optimization of the MLE may be rather unstable with long computation time if the starting values are far from the unknown true value, while the MLE numerically guided by the proposed initial estimators is much more stable, 
\rev{
implying that the numerical instability should be due to the local-maxima nature of the log-likelihood. 
Although not presented here, we found no instability when the true value of $\al$ was close to $1$ and the initial value for the numerical search of MLE was set to be the true values so that the local-maxima problem does not occur; this shows the effectiveness of using the 0-parametrization.
}

Having obtained the easy-to-compute initial estimator at rate $\sqrt{n}$, it would be natural from a theoretical point of view to consider the Newton-Raphson type one-step estimator (for example, \cite[Section 20]{Fer96}) to construct an asymptotically efficient estimator. However, we found that the numerical computation can be unstable and we do not pursue it in this paper; see also Remark \ref{hm:rem_onestep}.

\begin{table}[ht]
 \centering
  \begin{tabular}{ccccccc}
   \hline
   Method & $n$ &$\hat{\alpha}_n$ & $\hat{\beta}_n$ & $\hat{\sigma}_n$ & $\hat{\mu}_n$ & Time(s) \\
   \hline \hline
   Initial estimator & 300 & 0.58 & 0.79 & 1.08 & (4.63,1.77,3.08) & $7.1 \times 10^{-3}$ \\
   MLE (i) & 300 & 0.25 & -0.80 & 248.00 & (-4.45,-2.94,-2.61) & 5198.90 \\
   MLE (ii) & 300 & 0.85 & 0.45 & 1.53 & (5.02,1.98,3.04) & 500.20 \\
   \hline 
   True value && 0.8 & 0.5 & 1.5 & (5,2,3)\\
   \hline
  \end{tabular}
  
\caption{Parameter estimates and computational time (seconds) for the three cases.
For the initial values for numerical optimization of MLE, we used:
(i) $(1,0,1,(100,100,100))$ for the bad case where initial values are far from the true values;
(ii) $(0.58,0.79,1.08,(4.63, 1.77, 3.08))$ obtained by the proposed initial estimator: the CQMLE and the moment estimator.
}
  \label{tab:numerical}
\end{table}

\subsection{Outline and summary}

In Section \ref{sec:CQMLE}, we estimate the regression coefficient by the CQMLE without assuming that the errors are stable-distributed. We prove the asymptotic normality and the tail-probability estimate of the normalized CQMLE.
In Section \ref{sec:stable}, we estimate the stable distribution parameter $(\alpha,\beta,\sigma)$ by the method of moments through the residual sequence based on the CQMLE and showed the $\sqrt{n}$-consistency of $(\aes, \bes,\ses,\mes)$.
Importantly, it is essential in our proof that the CQMLE satisfies the polynomial-type tail-probability estimate.
As a consequence, we obtain a practical $\sqrt{n}$-consistent estimator, that can be computed in a flash as was seen in Table \ref{tab:numerical}.
In Section \ref{sec:numerical}, we extensively conduct numerical experiments, by drawing histograms and boxplots of the initial estimator when we estimated the regression coefficient by the CQMLE and, just for comparison, the other two estimators (the least-squares and the least absolute deviation estimators). Also shown are some histograms of the normalized MLE using the initial estimators as the starting value for optimization, from which it can be seen that we had no local optimum problem.

\section{CQMLE of location parameter}
\label{sec:CQMLE}

Throughout this section, we temporarily forget the assumption \eqref{hm:noise} on the noise distribution 
\rev{
so that the noise distribution $\mcl(\ep_1)$ may not be stable, 
and consider a broader setting for the noise distribution; see Assumption \ref{assump_AN_TPE}\eqref{assump_noise} together with Remark \ref{hm:rem_assump}\eqref{hm:rem_noise}. 
To remove the skewness of $\mcl(\ep_1)$, we take the difference in the model \eqref{hm:model}}:
\begin{equation}\label{hm:model_d}
\D_j Y = \D_j X \cdot \mu + \sig \D_j \ep,
\end{equation}
where $\D_j\zeta := \zeta_j - \zeta_{j-1}$ for $\zeta=X$, $Y$, and $\ep$.
\rev{Then, the random variables $\D_1 \ep,\D_2 \ep,\dots$ are not i.i.d. but form a $1$-dependent sequence. If in particular \eqref{hm:0stable.cf} holds, then $\D_1 \ep \sim S_{\al}^{0}(0, 2^{1/\al}\sig, 0)$.}

We define the CQMLE
by any element
\begin{equation}
\hat{\mu}_n \in \argmax_{\mu \in \Theta_{\mu}} H_n(\mu),
\nonumber
\end{equation}
where
\begin{equation}
H_n(\mu) := \sum_{j=2}^{n} \log \frac{1}{1 + (\Delta_jY - \Delta_jX \cdot \mu)^2}
\end{equation}
denotes the Cauchy quasi-likelihood (CQLF); the parameter space $\Theta_{\mu}$ of $\mu$ is now assumed to be a bounded convex domain.
Based on the random function $H_n(\mu)$, we will estimate the parameter $\mu$ as if the noise part ``$\sig \D_j \ep$'' in \eqref{hm:model_d} obeys the standard Cauchy distribution, although it may not be true. 
\rev{
It will turn out that, under appropriate conditions, the CQMLE $\hat{\mu}_n$ is asymptotically normal and satisfies the tail-probability estimate; the result can be regarded as a heavy-tail counterpart of the well-known fact that the Gaussian least-squares estimator can work for a broad class of finite-variance noise distributions, and seems new in the literature on linear regression with asymmetric heavy-tailed noise; the tail-probability estimate is essential in our proof of the $\sqrt{n}$-consistency of the moment estimator associated with the lower-order moments (Section \ref{sec:stable}).
The use of the Cauchy likelihood for possibly non-Cauchy noise was previously mentioned by, for example, \cite[Section 3]{CalDav98} in the context of linear processes with symmetric stable noise. 
}

To proceed, we introduce some assumptions. Let $A^{\otimes 2}:=AA^\top$.
\tcr{Recall that the true parameter value is given by $\theta_0 = (\al_0, \be_0, \sig_0, \mu_0)$.}

\begin{assump}\label{assump_AN_TPE}
~
\begin{enumerate}
    \item There exists a positive number $C_X$ such that $\sup_{j} |\Delta_jX|<C_X.$\label{X_assump}
    \item There exists a positive definite matrix $C_0$ and a constant $\be_1\in(0,1/2)$ such that\begin{equation}
    \sup_{n \in \mbbn} \left| n^{\beta_1}\left(\frac{1}{n} \sum_{j=2}^{n} (\Delta_j X)^{\otimes 2} - C_0 \right)\right| < \infty.
    \end{equation}\label{X^2_assump}
    \item There exists a matrix $D_0$ such that
    \begin{equation}
        \lim_{n \rightarrow \infty} \frac{1}{n-2}\sum_{j=2}^{n-1}(\Delta_j X)(\Delta_{j+1} X)^T = D_0.
    \end{equation}
    \label{XX_assump}
    
    \item \label{assump_eigen}
    There exists a continuous function $Y(\mu)$ such that
    \begin{equation}
        \sup_n \sup_{\mu \in \Theta_{\mu}}\left|n^{1/2-\beta_2}\left(
        Y^\star_n(\mu) -Y(\mu) \right)\right|<\infty
    \end{equation}
    for some constant $\beta_2\in[0,1/2)$, where
\begin{align}
    Y^\star_n(\mu)&:=-\frac{1}{n} \sum_{j=2}^{n} \E_{\tz}\big[\log(1 + \{\sig_0 \D_2 \ep - \Delta_jX\cdot(\mu-\mu_0)\}^{2}) 
    \nn\\
    &{}\qquad - \log(1 + (\sig_0 \D_2 \ep)^2)\big],
\end{align}
and there exists a constant $\lam>0$ such that $\tcr{Y(\mu)=Y(\mu)-Y(\mu_0)} \leq -\lambda|\mu -\mu_0|^2$ for all $\mu\in\mbbr^q$. 

\item \label{assump_noise}
$\E_{\tz}\left[|\log(1+\ep_1^2)|^K\right]<\infty$ for every $K>0$, and the symmetric distribution of the difference $\ep_2-\ep_1$ admits a Lebesgue density $\mathfrak{f}$ for which $\mathfrak{f}(1)>0$,
\begin{equation}
\sup_{x\ge 1}\mathfrak{f}(x)\le \mathfrak{f}(1) \le \inf_{0\le x\le 1}\mathfrak{f}(x),
\end{equation}
and there exist positive constants $\mathfrak{c}_1,\mathfrak{c}_2$ and $\mathfrak{b}$ for which $0<\mathfrak{c}_1<\mathfrak{c}_2<1$ and
\begin{equation}
    \inf_{\mathfrak{c}_1 \le x\le \mathfrak{c}_2}\mathfrak{f}(x) \ge (1+\mathfrak{b}) \mathfrak{f}(1).
\end{equation}
\end{enumerate}
\end{assump}

\medskip

Several remarks on Assumption \ref{assump_AN_TPE} are in order.

\begin{rem}
\label{hm:rem_assump}
~
\begin{enumerate}
    \item Note that Assumption \ref{assump_AN_TPE} excludes the constant factor (intercept) in ``$X_j\cdot 
\mu$''. This is the price we have to pay to handle possibly arbitrarily heavy-tailed and skewed $\ep_j$ 
by the simple and explicit CQLF $H_n(\mu)$. See Remark \ref{hm:rem_location} for a 
possible remedy for this inconvenience.
    \item Assumption \ref{assump_AN_TPE}\eqref{X^2_assump} and \eqref{XX_assump} are inevitable to identify the asymptotic covariance matrix of $\mes$ in Theorem \ref{hm:thm_CQMLE}.
In practice, one computes the quantities $\frac{1}{n} \sum_{j=2}^{n} (\Delta_j X)^{\otimes 2}$ and $\frac{1}{n}\sum_{j=2}^{n}(\Delta_j X)(\Delta_{j+1} X)^T$, and formally checks that the former is positive definite. Although we are assuming that $X_1,X_2,\dots$ are non-random, it could be, for example, 
a sequence of i.i.d. random vectors $(X_1,Y_1),(X_2,Y_2),\dots$ with $(X_j)_{j=1}^\infty$ being independent of $(\ep_j)_{j=1}^\infty$. Then, under the assumption that 
$\E_{\tz}[|X_1|^K]<\infty$ for sufficiently large $K>0$, one can apply the central limit theorem for one-dependent sequences \cite[Section 11]{Fer96} to conclude Assumption \ref{assump_AN_TPE}\eqref{X^2_assump} (with $\beta_1=1/2$) and \eqref{XX_assump} for $C_0=\E_{\tz}[(X_2-X_1)^{\otimes 2}]$ and $D_0=\E_{\tz}[(X_2-X_1)(X_3-X_2)^{\top}]$.
\item Assumption \ref{assump_AN_TPE}\eqref{assump_eigen} is a strong form of the identifiability (model-separation) condition. When $\{(X_j,Y_j)\}_{j=1}^n$ is i.i.d. as described in the previous item, then we can take $Y^\star_n(\mu)\equiv Y(\mu)$ with
\begin{align}
    Y(\mu) &= -\E_{\tz}\big[\log(1 + \{\sig_0 \D_2 \ep - \Delta_2 X\cdot(\mu-\mu_0)\}^{2}) 
    \nn\\
    &{}\qquad - \log(1 + (\sig_0 \D_2 \ep)^2)\big].
\end{align}

\item \label{hm:rem_noise}
In particular, Assumption \ref{assump_AN_TPE}\eqref{assump_noise} holds under \eqref{hm:noise} since $\mathfrak{f}(x)=c^{-1}\phi_{\al,0}(c^{-1}x)$ with $c:=2^{1/\al}\sig$.
More generally, it suffices that:
\begin{enumerate}
    \item the distribution of $\ep_1$ has at most Paretian tails;
    \item $\mathfrak{f}$ is bounded continuous and strictly decreasing on $[0,\infty)$. 
\end{enumerate}
Then, the logarithmic moments of any order are finite. We emphasize that the noise distribution $\mcl(\ep_1)$ is not necessary to be heavy-tailed and can also apply to a wide variety of light-tailed distributions as well.
\end{enumerate}   
\end{rem}

\begin{lem}\label{consistency}
Suppose that \eqref{X_assump}, \eqref{assump_eigen}, and \eqref{assump_noise} in Assumption \ref{assump_AN_TPE} hold. Then, we have the consistency
\begin{equation}
\hat{\mu}_n \cip \mu_0. \label{consistency-1}
\end{equation}
\end{lem}

\begin{proof}
Let $Y_n(\mu):= n^{-1} (H_n(\mu)-H_n(\mu_0))$; we set $n^{-1}$ just for notational brevity, though it is natural to set $(n-1)^{-1}$ instead.
Let $\del_{j}(\mu):=\Delta_jY - \Delta_jX \cdot \mu$; then, $\del_{j}(\mu_0)=\sig_0 \D_j \ep$.
We have
\begin{align}
Y_n(\mu) 
&= -\frac{1}{n}\sum_{j=2}^{n} \{ \log(1+ \{ \delta_j(\mu_0) -\Delta_jX \cdot (\mu-\mu_0)\}^2) - \log(1+ \delta_j(\mu_0)^2)\}\\
&=: -\frac{1}{n} \sum_{j=2}^{n} \xi_j(\mu)\\
&= -\frac{1}{\sqrt{n}} \left( \frac{1}{\sqrt{n}}\sum_{j=2}^{n} \{ \xi_j(\mu) - \E_{\tz}[\xi_j(\mu)]\}\right) 
+ \left(- \frac{1}{n} \sum_{j=2}^{n} E_{\tz}[\xi_j(\mu)]\right) \\
&=: - \frac{1}{\sqrt{n}} M_n(\mu) + Y^\star_n(\mu).
\label{Y_n}
\end{align}
Hence, under Assumption \ref{assump_AN_TPE}\eqref{assump_eigen},
\begin{equation}
\sup_{\mu \in \Theta_{\mu}} \left|n^{1/2-\beta_2}\left(Y_n(\mu) - Y(\mu)\right)\right|
\le C\left(n^{-\beta_2}\sup_{\mu \in \Theta_{\mu}}|M_n(\mu)| + 1\right)\quad\text{a.s.}
\end{equation}
for a universal constant $C>0$.
To conclude the consistency, it suffices to show 
\tcr{the following two conditions (see, for example, \cite[Theorem 5.7]{vdV98} for the general result for proving the consistency):
\begin{itemize}
    \item The moment estimate
\begin{equation}\label{M_n-2}
    \sup_{n} \E_{\tz}\left[\sup_{\mu \in \Theta_{\mu}}|M_n(\mu)|^{K_1}\right]<\infty
\end{equation}
for $K_1>0$;
    \item The identifiability $\sup_{\mu : d(\mu, \mu_0) \geq \varepsilon} Y(\mu) < Y(\mu_0)=0$ for each $\ep>0$.
\end{itemize}
The former implies that the uniform convergence $\sup_{\mu \in \Theta_{\mu}}|Y_n(\mu) - Y(\mu)| \cip 0$ and that the latter follows from Assumption \ref{assump_AN_TPE}\eqref{assump_eigen}.} 
The estimate \eqref{M_n-2} is more than necessary for consistency, but it will be used later in Theorem \ref{hm:thm_CQMLE}.

\tcr{It remains to show \eqref{M_n-2}.}
Since $\Theta_{\mu}$ is a bounded convex domain, 
\tcr{we can apply Sobolev's inequality \cite[Remark in p.415]{Ada73} to conclude that for $K_1>q$,
\begin{equation}
\sup_{\mu \in \Theta_{\mu}} |M_n(\mu)|^{K_1}
\lesssim 
\int_{\Theta_{\mu}} |M_n(\mu)|^{K_1} d\mu 
+ 
\int_{\Theta_{\mu}} |\p_\mu M_n(\mu)|^{K_1} d\mu\qquad \text{a.s.},
\end{equation}
}
where the symbol $a_n\lesssim b_n$ for two possibly random non-negative sequences $(a_n)$ and $(b_n)$ means that there exists a universal constant $C>0$ for which $a_n\le C b_n$ a.s. 
\tcr{Thus, taking the expectation and applying Fubini's theorem,}
\begin{align}
\E_{\tz}\left[\sup_{\mu \in \Theta_{\mu}}|M_n(\mu)|^{K_1}\right]
& \lesssim \sup_{\mu \in \Theta_{\mu}} \E_{\tz}[|M_n(\mu)|^{K_1}] + \sup_{\mu \in 
\Theta_{\mu}} \E_{\tz}[|\p_{\mu} M_n(\mu)|^{K_1}].
\label{M_n}
\end{align}
\tcr{We will check the finiteness of the two expectations on the right-hand side for $K_1>q\vee 2$. 
Without loss of generality, we may and do suppose that $n$ is even, say $n=2m$. Recall that $M_n(\mu)$ involves the sum of $1$-dependent centered random variables $\tilde{\xi}_j(\mu):=\xi_j(\mu)- \E_{\tz}[\xi_j(\mu)]$, $1\le j\le n$. We observe that
\begin{equation}
    |M_n(\mu)|^{K_1} 
    \lesssim \left|\frac{1}{\sqrt{m}} \sum_{k=1}^{m} \tilde{\xi}_{2l}(\mu)\right|^{K_1} + \left|\frac{1}{\sqrt{m}} \sum_{l=1}^{m-1} \tilde{\xi}_{2l+1}(\mu)\right|^{K_1},
\end{equation}
the two sums on the right-hand side consisting of independent centered random variables. For the first sum, Burkholder's and Jensen's inequalities give
\begin{align}
& \sup_{\mu \in \Theta_{\mu}}\E_{\tz}\left[ \left|\frac{1}{\sqrt{m}} \sum_{k=1}^{m} \tilde{\xi}_{2l}(\mu)\right|^{K_1}\right]
\nn\\
&\qquad \lesssim \sup_{\mu \in \Theta_{\mu}} \E_{\tz}\left[ \left( \frac{1}{m}\sum_{l=1}^{m} \tilde{\xi}_{2l}(\mu)^2\right)^{K_1/2}\right]\\
&\qquad \lesssim \sup_{\mu \in \Theta_{\mu}} \frac{1}{m} \sum_{l=1}^{m} \E_{\tz}[|\xi_{2l}(\mu)|^{K_1}]\\
&\qquad \lesssim  \sup_{\mu \in \Theta_{\mu}} \frac{1}{m} \sum_{l=1}^{m} \E_{\tz}[| \log(1+ \del_{2l}(\mu)^2) |^{K_1}]\\
&\qquad \lesssim \sup_{\mu \in \Theta_{\mu}} \frac{1}{m} \sum_{l=1}^{m} \E_{\tz}[| \log(1+ (\del_2(\mz) + \Delta_{2l}X\cdot(\mz - \mu))^2) |^{K_1}]\\
&\qquad \lesssim 1+ \E_{\tz}[| \log(1+ (|\ep_{1}|+ |\ep_2| + 1)^2) |^{K_1}]\\
&\qquad \lesssim 1 + \E_{\tz}\left[|\log(1+\ep_1^2)|^{K_1}\right] \lesssim 1.
\end{align}
The second sum can be handled in the same way:
\begin{equation}
    \sup_{\mu \in \Theta_{\mu}}\E_{\tz}\left[ \left|\frac{1}{\sqrt{m}} \sum_{l=1}^{m-1} \tilde{\xi}_{2l+1}(\mu)\right|^{K_1}\right] \lesssim 1.
\end{equation}
Thus, we have shown that the first term of \eqref{M_n} is bounded in $n$. The second term can be handled similarly.
}
We conclude \eqref{M_n-2}, hence the proof is complete.
\end{proof}

\tcr{
Let $\del_{j}(\mu_0)=\sig_0 \D_j \ep=:S_j$, and let
\begin{align}
\Sigma_0 &:= 
4\E_{\tz}\left[\left( \frac{S_2}{1 + S_2^2} \right)^2\right]
C_0 + 4\E_{\tz}\left[\left(\frac{S_2}{1+ S_2^2}\right)\left(\frac{S_3}{1+ S_3^2}\right)\right](D_0+D_0^{T}),
\nn\\
\Gam_0 &:= 2 \E_{\tz}\left[\frac{1 - S_2^2}{(1+S_2^2)^2}\right] C_0,
\nonumber
\end{align}
both of which exist and are finite since the random variables inside $\E_{\tz}$ are essentially bounded.
}

\begin{thm}\label{hm:thm_CQMLE}
Suppose that Assumption \ref{assump_AN_TPE} holds. Then, we have
\begin{equation}
\label{hm:mu-AN}
\sqrt{n} (\hat{\mu}_n - \mu_0) \cil N(0, \Gamma_0^{-1}\Sigma_0 \Gamma_0^{-1}).
\end{equation}
Furthermore, for any $L>0$ the following tail-probability estimate holds
\begin{equation}\label{hm:mu-TPE}
\sup_{r>0}\sup_{n \in \mbbn} \pr_{\tz}\left[|\sqrt{n}(\hat{\mu}_n - \mz)|>r\right]\, r^{L}< \infty.
\end{equation}
\end{thm}

\begin{proof}
First, we prove \eqref{hm:mu-AN}.
From the mean value theorem,
\begin{equation}
\p_{\mu}H_n(\hat{\mu}_n) = \p_{\mu}H_n(\mz) + \int_{0}^{1} \p_{\mu}^2 H_n(\mz + t(\hat{\mu}_n-\mz))dt (\hat{\mu}_n - \mz).
\label{hm:add1}
\end{equation}
By the consistency of $\mes$, we may and do suppose that $\p_{\mu}H_n(\hat{\mu}_n)=0$, so that
\begin{equation}
\sqrt{n}(\hat{\mu}_n - \mz) = \left( -\frac{1}{n} \int_{0}^{1} \p_{\mu}^2 H_n(\mz + t(\hat{\mu}_n-\mz))dt\right)^{-1} \frac{1}{\sqrt{n}} \p_{\mu}H_n(\mz).
\end{equation}
We can conclude \eqref{hm:mu-AN} by showing
\begin{align}
\frac{1}{\sqrt{n}}\p_{\mu}H_n(\mz) \cil N(0, \Sigma_0),
\label{hm:add2}\\
-\int_{0}^{1}  \frac{1}{n} \p_{\mu}^2 H_n(\mz + t(\hat{\mu}_n-\mz)) dt
&\cip \Gamma_0,
\label{hm:add3}
\end{align}
with $\Gamma_0$ being positive definite.

\medskip

\noindent
\textit{Proof of \eqref{hm:add2}.}
Observe that
\begin{align}
\frac{1}{\sqrt{n}}\p_{\mu} H_n(\mz)
&= -\frac{1}{\sqrt{n}} \sum_{j=2}^{n} \p_{\mu} \log(1+ (\Delta_jY-\Delta_jX \cdot \mu)^2)|_{\mu=\mz}\\
&= \sum_{j=2}^{n} \frac{2}{\sqrt{n}} \frac{S_j}{1+ S_j^2} 
\Delta_jX =:\sum_{j=2}^{n} \chi_{n,j}.
\end{align}
Since the distribution of $S_j = \sig_0\Delta_j \ep$ is symmetric,
\begin{align}
\E_{\tz}[\chi_{n,j}]
&= \frac{2}{\sqrt{n}} \E_{\tz} \left[\frac{S_2}{1+ S_2^2} \right]\Delta_jX =0.
\end{align}
We also have
\begin{align}
& \lim_{n \rightarrow \infty} \E_{\tz}\left[\left(\sum_{j=2}^{n}\chi_{n,j}\right)^{\otimes 2}\right]
\nn\\
&= \lim_{n \rightarrow \infty} \frac{4}{n} \E_{\tz}\left[\sum_{i=2}^{n}\sum_{j=2}^{n} \left(\frac{S_i}{1+ S_i^2}\right) \left(\frac{S_j}{1+ S_j^2}\right) (\Delta_iX)(\Delta_jX)^{T}\right]\\
&=\lim_{n \rightarrow \infty} \frac{4}{n} \sum_{j=2}^{n} \E_{\tz}\left[\left(\frac{S_2}{1+ S_2^2}\right)^2 \right] (\Delta_jX)^{\otimes2}\\
&{}\qquad + 4\E_{\tz}\left[\left(\frac{S_2}{1+ S_2^2}\right)\left(\frac{S_3}{1+ S_3^2}\right)\right]\\
&{}\qquad  \times \bigg \{\lim_{n \rightarrow \infty} \frac{1}{n} \sum_{j=2}^{n-1} \{ (\Delta_jX)(\Delta_{j+1}X)^{T} + (\Delta_{j+1}X)(\Delta_{j}X)^{T} \} \bigg\}\\
&=\Sigma_0.
\end{align}
From the above discussion and \tcr{the central limit theorem for sequences of $1$-dependent random variables \cite[Theorem 11]{Fer96}}, we get \eqref{hm:add2}.

\medskip

\noindent
\textit{Proof of \eqref{hm:add3}.}
By Taylor's theorem,
\begin{align}
& -\frac{1}{n} \int_{0}^{1} \p_{\mu}^2H_n(\mz + t(\hat{\mu}_n-\mz))dt \nn\\
&= -\frac{1}{n} \p_{\mu}^2H_n(\mz) - \frac{1}{n} \int_{0}^{1} \int_{0}^{1} t \p_{\mu}^3H_n(\mz + ut(\hat{\mu}_n-\mz)) dudt (\hat{\mu}_n - \mz),
\end{align}
By the consistency of $\hat{\mu}_n$ proved in Lemma \ref{consistency} and since $n^{-1}\sup_{\mu \in \Theta_{\mu}} \left|\p_{\mu}^3H_n(\mu)\right| \lesssim 1$ a.s., we have
\begin{align}
\left| \frac{1}{n} \int_{0}^{1} \int_{0}^{1} t\p_{\mu}^3 H_n(\mz + ut(\hat{\mu}_n - \mz)) dudt (\hat{\mu}_n - \mz) \right|
&\leq \frac{1}{n} \sup_{\mu \in \Theta_{\mu}} \left|\p_{\mu}^3 H_n(\mu)\right| |\hat{\mu}_n - \mz|\\
&= O_p(1) o_p(1) = o_p(1).
\end{align}
It follows that
\begin{align}
-\frac{1}{n} \int_{0}^{1} \p_{\mu}^2 H_n(\mz + t(\hat{\mu}_n-\mz))dt
&= -\frac{1}{n} \p_{\mu}^2H_n(\mz) + o_p(1).
\end{align}
We have
\begin{align}
-\frac{1}{n} \p_{\mu}^2H_n\left(\mz \right)
&= \frac{1}{n} \p_{\mu}^2 \sum_{j=2}^{n} \log(1+(\Delta_jY - \Delta_jX \cdot \mu)^2)\bigg|_{\mu=\mz}\\
&= \frac{1}{n} \sum_{j=2}^{n} \frac{2 - 2S_j^2}{(1+S_j^2)^2}(\Delta_jX)^{\otimes 2}\\
&= \frac{1}{n} \sum_{j=2}^{n} \left( \frac{2 - 2S_j^2}{(1+S_j^2)^2}- \E_{\tz}\left[\frac{2 - 2S_2^2}{(1+S_2^2)^2}\right]\right)(\Delta_jX)^{\otimes 2}
\nn\\
&{}\qquad + \E_{\tz}\left[\frac{2 - 2S_2^2}{(1+S_2^2)^2}\right] \frac{1}{n} \sum_{j=2}^{n} (\Delta_jX)^{\otimes 2}.
\end{align}
Using the Lindeberg-Lyapunov theorem as before, the first term on the rightmost side is $O_p(n^{-1/2})$.
Thus, we obtain
\begin{equation}
-\frac{1}{n} \p_{\mu}^2 H_n\left(\mz \right) \cip \E_{\tz}\left[\frac{2 - 2S_2^2}{(1+S_2^2)^2}\right]C_0 = \Gam_0,
\end{equation}
followed by \eqref{hm:add3}.
Finally, $\Gamma_0 >0$ holds since the expectation is positive by Lemma \ref{lem:integral} below.

\medskip

Next, we turn to the proof of \eqref{hm:mu-TPE}, the tail-probability estimate of $\hat{\mu}_n$. 
Recall the notation $Y_n(\mu)= \frac{1}{n}(H_n(\mu)- H_n(\mu_0))$ 
and let $\Gamma_n(\mu)= -\frac{1}{n} \p_{\mu}^2 H_n(\mu)$.
\tcr{By Assumption \ref{assump_AN_TPE}, the conditions $[A4'],[B1]$, and $[B2]$ in \cite[Theorem 3(c)]{Yos11} are satisfied. Therefore, it remains to verify conditions $[A1'']$ and $[A6]$. To this end, it suffices to establish the following for any $K>q \vee 2$:}
  \begin{align}
      & \sup_{n \in \mbbn} \E_{\tz}\left[\left(\frac{1}{n}\sup_{\mu \in \Theta_{\mu}}|\p_{\mu}^{3}H_n(\mu)|\right)^{K}\right]
      +\sup_{n \in \mbbn} \E_{\tz}\left[\left(n^{\beta_1}|\Gamma_n(\mu_0)-\Gamma_0|\right)^{K}\right]
      \nn\\
      &{}\qquad +\sup_{n \in \mbbn} \E_{\tz}\left[\left|\frac{1}{\sqrt{n}}\p_{\mu} H_n(\mz)\right|^{K}\right]
      \nn\\
      &{}\qquad 
       +\rev{\sup_{n \in \mbbn} \E_{\tz}\left[\left(\sup_{\mu \in \Theta_{\mu}}n^{\frac{1}{2}-\beta_2}|Y_n(\mu)-Y(\mu)|\right)^{K}\right]} 
      < \infty.
      \label{tail_prob_condi}
  \end{align}
 where $\beta_1\in(0,1/2)$ is the one given in Assumption \ref{assump_AN_TPE}\eqref{X^2_assump} and $\be_2 \in (0,1/2)$ is the one given in Assumption \ref{assump_AN_TPE}\eqref{assump_eigen}.

We begin with the first term. 
\rev{Let $f(x):= \log(1+x^2)$.}
Then, $\sup_{x \in \mbbr} |\p_{x}^kf(x)| \lesssim 1$ ($k\ge 1$), so that following inequalities hold:
\begin{align}
\sup_{n \in \mbbn} \E_{\tz}\left[\left(\frac{1}{n}\sup_{\mu \in \Theta_{\mu}}|\p_{\mu}^{3}H_n(\mu)|\right)^{K}\right]
&\leq \sup_{n \in \mbbn} \E_{\tz}\left[\left(\frac{1}{n}\sum_{j=2}^{n} \sup_{\mu \in \Theta_{\mu}} |\p_{\mu}^3 f(\del_j(\mu))|\right)^{K}\right]
\lesssim 1.
\nonumber
\end{align}
Turning to the second term.
We observe that
\begin{flalign}
\Gamma_n(\mu_0)
&= -\frac{1}{n} \p_{\mu}^2H_n(\mu_0)\\
&= \frac{1}{n}\sum_{j=2}^{n}\p_{x}^2f(S_j) (\Delta_jX)^{\otimes2}\\
&= \E_{\tz}[\p_{x}^2f(S_2)]\frac{1}{n}\sum_{j=2}^{n} (\Delta_jX)^{\otimes2}
+ \frac{1}{n}\sum_{j=2}^{n} \{\p_{x}^2f(S_j) - \E_{\tz}[\p_{x}^2f(S_2)]\} (\Delta_jX)^{\otimes2}\\
&= \Gamma_0 +  \E_{\tz}[\p_{x}f^2(S_2)]\left(\frac{1}{n}\sum_{j=2}^{n} (\Delta_jX)^{\otimes2}-C_0 \right)  \\
&{}\qquad+\frac{1}{n}\sum_{j=2}^{n} \{\p_{x}^2f(S_j) - \E_{\tz}[\p_{x}^2 f(S_2)] \} (\Delta_jX)^{\otimes2},
\label{Gamma_n}
\end{flalign}
where we used $\Gamma_0 = \E_{\tz}[\p_{x}^2f(S_2)] C_0$.
By rewriting the above equations,
\begin{align}
n^{\be_1}(\Gamma_n(\mu_0)- \Gamma_0)
&= \E_{\tz}[\p_{x}^2f(S_2)] \left(\frac{1}{n}\sum_{j=2}^{n} (\Delta_jX)^{\otimes2}-C_0 \right)n^{\be_1} 
\\
&{}\qquad+ \frac{1}{n}\sum_{j=2}^{n} \{\p_{x}^2f(S_j) - \E_{\tz}[\p_{x}^2f(S_2)]\} (\Delta_jX)^{\otimes2}n^{\be_1}.
\end{align}
Since $\p_{x}^2f(x)$ is bounded,
for $K \geq 2$, Burkholder's inequality yields
\begin{align}
& \E_{\tz}\left[|n^{\be_1}(\Gamma_n(\mu_0)- \Gamma_0)|^{K}\right]
\nn\\
&\lesssim \left| 
\left(\frac{1}{n}\sum_{j=2}^{n}
(\Delta_jX)^{\otimes2}-C_0 \right)n^{\be_1} \right|^{K}\\
&{}\quad+ n^{K(\be_1-1/2)} 
\E_{\tz} \left[ \left|\frac{1}{\sqrt{n}}\sum_{j=2}^{n}
\{\p_{x}^2f(S_j) - \E_{\tz}[\p_{x}^2f(S_2)]\} (\Delta_jX)^{\otimes2}
\right|^{K}\right]\\
&\lesssim 1+n^{K(\be_1-1/2)} \lesssim 1.
\end{align}


As for the third term of the inequality \eqref{tail_prob_condi}, it is obvious that
\begin{align}
\E_{\tz}[\p_{x}f(S_j) \Delta_jX]
&= \E_{\tz} \left[\frac{2S_2}{1+ S_2^2} \right]  \Delta_jX= 0.
\end{align}
Since $\p_{x}f(S_2)\cdot \Delta_2 X,\dots,\p_{x}f(S_n)\cdot \Delta_n X$ are $1$-dependent, the following inequality holds for $K \geq 2$ from Burkholder's inequality:
\begin{align}
\E_{\tz} \left[\left|\frac{1}{\sqrt{n}}\p_{\mu}H_n(\mz)\right|^{K} \right]
&\lesssim \E_{\tz}\left[\left|\frac{1}{\sqrt{n/2}}\sum_{j=1}^{[n/2]} \p_{x}f(S_{2j}) \Delta_{2j}X\right|^{K} \right]\\
&{}\quad+ \E_{\tz}\left[\left|\frac{1}{\sqrt{n/2}}\sum_{j=1}^{[(n-1)/2]} \p_{x}f(S_{2j+1})\Delta_{2j+1}X \right|^{K} \right]
\lesssim 1.
\end{align}
Finally, we already proved in the proof of Lemma \ref{consistency} that the fourth term of the inequality \eqref{tail_prob_condi} is finite.
The proof of \eqref{tail_prob_condi} is complete.
\end{proof}

\begin{lem}\label{lem:integral}
Under Assumption \ref{assump_AN_TPE}\eqref{assump_noise}, we have
\begin{equation}
    \E_{\tz} \left[ \frac{1-S_2^2}{(1+S_2^2)^2} \right] > 0.
\end{equation}
\end{lem}
\begin{proof}
The function $\psi(x):= \frac{1-x^2}{(1+x^2)^2}$ is symmetric, is positive (resp. negative) for $|x|\leq1$ (resp. $|x|>1$), and satisfies that $\int_{\mbbr}\psi(x)dx =0$ (apply the change of variable $x=\tan \alpha$). 
We have $\int_{1}^\infty \psi(x) \mathfrak{f}(x)dx \ge \mathfrak{f}(1)\int_1^\infty \psi(x)dx$ and also $\int_{0}^{1} \psi(x) \mathfrak{f}(x)dx > \mathfrak{f}(1) \int_{0}^{1} \psi(x) dx$; the latter inequality is strict since
\begin{equation}
    \int_{0}^{1} \psi(x) (\mathfrak{f}(x) - \mathfrak{f}(1))dx \ge \int_{\mathfrak{c}_1}^{\mathfrak{c}_2} \psi(x) (\mathfrak{f}(x) - \mathfrak{f}(1))dx \ge \mathfrak{b} \mathfrak{f}(1)\,\int_{\mathfrak{c}_1}^{\mathfrak{c}_2} \psi(x)dx>0.
\end{equation}
Thus,
\begin{align*}
\E_{\tz} \left[\psi(S_2) \right] &= 2\int_{0}^\infty \psi(x) \mathfrak{f}(x)dx
> 2\mathfrak{f}(1)\int_{0}^{1} \psi(x)dx + 2\mathfrak{f}(1)\int_{1}^{\infty}\psi(x)dx
\nn\\
&=2\mathfrak{f}(1) \int_{\mbbr}\psi(x)dx = 0.
\end{align*}
\end{proof}

\begin{rem}
\label{hm:rem_LADE}
Instead of the CQMLE, one may think of the least absolute deviation estimator (LADE) as was studied in \rev{\cite{Kni98} and \cite{Hos22_mt}}. 
\rev{
However, because of some technical issues caused by the non-differentiability of the quasi-likelihood associated with the LADE, we currently have not proved the tail-probability estimate for the LADE, that will be required to show the ($\sqrt{n}$-)consistency of the moment estimator considered in Section \ref{sec:stable}; therefore, we do not know if we can follow a similar route to the CQMLE case.}
For comparison purposes, we will also observe the finite-sample performance of the LADE in the numerical experiments in Section \ref{sec:numerical}.
\end{rem}

\rev{
\begin{rem}
One may think of the joint estimation of the regression coefficient and the noise scale parameter through the CQLF. However, as can be seen from \cite[Proposition 1]{FanQiXiu14}, the estimation of the noise-scale parameter through the CQLF will lead to an inconsistent estimator. This is the reason why we used the CQLF only for estimating the regression coefficient $\mu$.
\end{rem}
}


\section{Method of moments based on residuals}
\label{sec:stable}

Let us go back to the setup described in Section \ref{hm:sec_st.reg.setup}.
The next step is to estimate the parameters $(\al, \beta, \sig)$ characterizing the noise distribution $\mcl(\ep_1)$.
We introduce the residual sequence
\begin{equation}
    \hat{\ep}_{j} := Y_j - X_j \cdot \mes
\end{equation}
with $\mes$ denoting the CQMLE.
We will estimate the parameters by the method of moments, based on the residuals
\begin{align*}
    \hat{\ep}_{S,j}:= \hat{\ep}_{j}- \hat{\ep}_{j-1},\qquad
\hat{\ep}_{C,j}:= \hat{\ep}_{j}+\hat{\ep}_{j-1}-2\hat{\ep}_{j-2}
\end{align*}
for $\ep_{S,j}:=\ep_j-\ep_{j-1}$ $(j=2,\dots,n)$ and $\ep_{C,j}:=\ep_j+\ep_{j-1}- \ep_{j-2}$ $(j=3,\dots,n)$, respectively.
\rev{
By manipulating the characteristic function (see \eqref{hm:0stable.cf} and \eqref{hm:0stable.cf_al=1}), we see that
\begin{align}
 \ep_{S,j} &\sim S_{\al}^{0}(0, 2^{1/\al}\sig, 0), \nn\\
 \ep_{C,j} &\sim S_{\al}^{0} \left( \left( \frac{2-2^{\al}}{2+2^{\al}}\right)\be, (2+ 2^{\al})^{1/\al}\sig, \left( \frac{2-2^{\al}}{2+2^{\al}}\right)\be (2+ 2^{\al})^{1/\al}\sig \tan \frac{\al \pi}{2}\right).
 \nn
\end{align}
}
More specifically, we will prove the $\sqrt{n}$-consistency of the moment estimator.
Introduce the following notation:
\begin{align}
\tz^{-} &=  (\al_0, \beta_0, \sig_0), \quad \tz = (\al_0, \beta_0, \sig_0, \mu_0),\quad \rev{|x|^{\la r \ra}:= \sgn(x)|x|^{r}\quad(x\in\mbbr)},\\
%
%
m_n &:= 
\bigg(
\frac{1}{n-1} \sum_{j=2}^{n} |\hat{\ep}_ {S,j}|^{r},\,
\frac{1}{n-1} \sum_{j=2}^{n}|\hat{\ep}_ {S,j}|^{2r},\,
\nn\\
&{}\qquad 
\frac{1}{n-2} \left. \sum_{j=3}^{n}|\hat{\ep}_ {C,j}|^{\la r \ra} \middle/ {\frac{1}{n-2}\sum_{j=3}^{n}|\hat{\ep}_ {C, j}|^{r}} \right.
\bigg),
\nn\\
m(\tz^{-}) &:= \left(
\E_{\tz}[|\ep_{S,2}|^{r}],\,
\E_{\tz}[|\ep_{S,2}|^{2r}],\,
\left. \E_{\tz}[|\ep_{C,3}|^{\la r \ra}] \right/ \E_{\tz}[|\ep_{C,3}|^{r}]
\right),
\end{align}
where $r$ is a given constant and satisfies 
\begin{equation}\label{hm:r_region}
0 <r<\frac{\az}{4},
\end{equation}
which in particular implies that $r<1/2$; 
the quantities $m_n$ and $m(\tz^{-})$ were also considered in \cite[Section 3.3]{Hos22_mt} without taking care of the rate of convergence of the associated moment estimator.


\textit{Throughout this section, Assumption \ref{assump_AN_TPE} refers to the one in Section \ref{sec:CQMLE} with the additional requirement that $\ep_1,\ep_2,...\sim \text{i.i.d.}~S_\al(\beta,1,0)$; see also Remark \ref{hm:rem_assump}\eqref{hm:rem_noise}.}
The following lemma will be used to establish the $\sqrt{n}$-consistency of the estimator of 
$(\al, \beta, \sig)$.

\begin{lem}\label{moment_tight}
\label{tightness} Under Assumption \ref{assump_AN_TPE},
\begin{equation}
\sqrt{n}\left\{m_n - m(\tz^{-})\right\} = O_p(1).
\end{equation}
\end{lem}

\begin{proof}
We will complete the proof by showing the component-wise tightness of $\sqrt{n}\left\{m_n - m(\tz^{-})\right\}$:
\begin{align}
& \sqrt{n}\left(\frac{1}{n-1} \sum_{j=2}^{n}|\hat{\ep}_ {S,j}|^{r}- \E_{\tz}[|\ep_{S,2}|^r]\right)= O_p(1),\label{eq:S,r}\\
& \sqrt{n}\left(\frac{1}{n-1} \sum_{j=2}^{n}|\hat{\ep}_ {S,j}|^{2r}- \E_{\tz}[|\ep_{S,2}|^{2r}]\right) = O_p(1),\label{eq:S,2r}\\
& \left. \sqrt{n}\left(\frac{1}{n-2} \sum_{j=3}^{n}|\hat{\ep}_ {C, j}|^{\la r \ra} \middle/ {\frac{1}{n-2}\sum_{j=3}^{n}|\hat{\ep}_ {C, j}|^{r}} \right.- \left. \E_{\tz}[|\ep_{C,3}|^{\la r \ra}] \middle/  \E_{\tz}[|\ep_{C,3}|^{r}]\right. \right)=O_{p}(1). \label{eq:C,<r>}
\end{align}

We begin with \eqref{eq:S,r}. 
Note that
\begin{align}
& \sqrt{n} \left|\frac{1}{n-1} \sum_{j=2}^{n}|\hat{\ep}_ {S,j}|^{r}- \E_{\tz}[|\ep_{S,2}|^r] \right|
\nn\\
&= \sqrt{n} \left| \frac{1}{n-1} \sum_{j=2}^{n}|\ep_ {S,j}|^{r}- \E_{\tz}[|\ep_{S,2}|^r] + \frac{1}{n-1} \sum_{j=2}^{n}(|\hat{\ep}_{S,j}|^{r}- |\ep_{S,j}|^{r}) \right|\\
&\leq \sqrt{n} \left|\frac{1}{n-1} \sum_{j=2}^{n}|\ep_ {S,j}|^{r}- \E_{\tz}[|\ep_{S,2}|^r] \right| + \sqrt{n}\frac{1}{n-1} \sum_{j=2}^{n}
\left| |\hat{\ep}_{S,j}|^{r}- |\ep_{S,j}|^{r} \right|\\
&= O_p(1)+\sqrt{n}\frac{1}{n-1} \sum_{j=2}^{n}
\left| |\hat{\ep}_{S,j}|^{r}- |\ep_{S,j}|^{r} \right|,
\end{align}
where we used the $1$-dependent central limit theorem \cite[Theorem 11]{Fer96} at the last step. Let 
\begin{equation}
I(A):= \begin{cases}
1 & \text{(if $A$ is true)}\\
0 & \text{(if $A$ is false)}.
\end{cases}
\end{equation}
For the second term in the upper bound, we have
\begin{align}
& \sqrt{n}\frac{1}{n-1} \sum_{j=2}^{n}| |\hat{\ep}_{S,j}|^{r}- |\ep_{S,j}|^{r}|
\nn\\
&=\sqrt{n}\bigg\{\frac{1}{n-1} \sum_{j=2}^{n}| |\hat{\ep}_{S,j}|^{r}- |\ep_{S,j}|^{r}|I\left(|\hat{\ep}_{S,j}-\ep_{S,j}|\geq\frac{1}{2}|\ep_{S,j}|\right)\\
& \qquad +\frac{1}{n-1} \sum_{j=2}^{n}| |\hat{\ep}_{S,j}|^{r}- |\ep_{S,j}|^{r}|I\left(|\hat{\ep}_{S,j}-\ep_{S,j}|<\frac{1}{2}|\ep_{S,j}|\right)\bigg\}\\
&=:\sqrt{n}(a_n+b_n).
\end{align}
We will separately prove
\begin{equation}
\sqrt{n}\,a_n=O_p(1)\label{an_tight},
\end{equation}
\begin{equation}
\sqrt{n}\,b_n=O_p(1)\label{bn_tight}.
\end{equation}
Let $\hat{u}_{\mu,n}:=\sqrt{n}(\hat{\mu}_n - \mz)$.
From Assumption \ref{assump_AN_TPE},
\begin{align}
a_n
&\leq \frac{1}{n-1} \sum_{j=2}^{n} |\hat{\ep}_{S,j}- \ep_{S,j}|^{r}I\left(|\Delta_jX|\cdot|\mes-\mu_0|\geq\frac{1}{2}|\ep_{S,j}|\right)\\
&\leq \frac{1}{n-1} \sum_{j=2}^{n} |\Delta_jX|^{r}\cdot|\mes-\mu_0|^{r}I\left(|\Delta_jX|\cdot|\mes-\mu|\geq\frac{1}{2}|\ep_{S,j}|\right)\\
&\lesssim 
|\hat{u}_{\mu,n}|^{r} n^{-\frac{r}{2}} \frac{1}{n-1} \sum_{j=2}^{n} I\left(|\hat{u}_{\mu,n}|\geq \frac{\sqrt{n}}{2}\frac{|\ep_{S,j}|}{C_X}\right)\\
&=: |\hat{u}_{\mu,n}|^{r} n^{-\frac{r}{2}} \frac{1}{n-1} \sum_{j=2}^{n} I(B_{n,j})
=O_p(n^{-\frac{r}{2}}) \frac{1}{n-1} \sum_{j=2}^{n} I(B_{n,j}).
\end{align}
To derive \eqref{an_tight}, it suffices to prove
\begin{equation}
    \E_{\tz}\left[n^{(1-r)/2} \frac{1}{n-1} \sum_{j=2}^{n} I(B_{n,j})\right]=O(1).
\end{equation}
We can take a positive constant $0<\ep<1/2$ and $L>0$ such that $(1-r)/2 \leq \min\{L(1/2-\ep),\, \ep\}$. 
Then, using the inclusion relation
\begin{align}
B_{n,j}
&\subset \{|\hat{u}_{\mu,n}|> n^{\frac{1}{2}-\ep}\} \cup \left \{ |\ep_{S,,j}| \leq \frac{2C_X n^{\frac{1}{2}-\ep}}{\sqrt{n}}\right\}.
\nn
\end{align}
\tcr{Since $\ep_{S,j}$ follows a stable distribution with a bounded continuous density function, we have}
\begin{align}
&\E_{\tz}\left[n^{\frac12-\frac{r}{2}} \frac{1}{n-1} \sum_{j=2}^{n} I(B_{n,j}) \right]\\
&\le n^{\frac12-\frac{r}{2}} \left(\frac{1}{n-1} \sum_{j=2}^{n} \pr_{\tz}[|\hat{u}_{\mu,n}|> n^{\frac{1}{2}-\ep}] + 
  \frac{1}{n-1} \sum_{j=2}^{n} \pr_{\tz}[|\ep_{S,j}| \leq 2C_Xn^{-\ep}]\right) \\
&\lesssim n^{\frac12-\frac{r}{2}}( n^{-L(\frac{1}{2}-\ep)}+n^{-\ep} ) \lesssim 1.
\end{align}
This concludes \eqref{an_tight}.
To prove \eqref{bn_tight}, we apply Taylor's theorem:
\begin{align}
b_n
&=\frac{1}{n-1} \sum_{j=2}^{n}\left| \int_{0}^{1}r \{ |\ep_{S,j}| + u(|\hat{\ep}_{S,j}|- |\ep_{S,j}|)\}^{r-1}du (|\hat{\ep}_{S,j}| - |\ep_{S,j}|) \right| 
\nn\\
&{}\qquad\times I\left(|\hat{\ep}_{S,j}-\ep_{S,j}|<\frac{1}{2}|\ep_{S,j}|\right)\\
&\lesssim \frac{1}{n-1} \sum_{j=2}^{n}
|\ep_{S,j}|^{r-1} |\Delta_jX|| \hat{\mu}_n - \mz|I\left(|\hat{\ep}_{S,j}-\ep_{S,j}|<\frac{1}{2}|\ep_{S,j}|\right)\\
&\lesssim \frac{1}{n-1} \sum_{j=2}^{n}
|\ep_{S,j}|^{r-1} |\hat{\mu}_n - \mz|
\nn\\
&\lesssim 
n^{-1/2}\frac{1}{n-1} \sum_{j=2}^{n}|\ep_{S,j}|^{r-1}|\hat{u}_{\mu,n}| = O_p(n^{-\frac{1}{2}}).
\end{align}
Here, we used the fact that $\E[|\ep_{S,2}|^{r-1}]<\infty$ in the last step. 
Thus, we obtained \eqref{eq:S,r}, and we can show \eqref{eq:S,2r} in the same way as in \eqref{eq:S,r}.

\medskip

It remains to prove \eqref{eq:C,<r>}. Similar to \eqref{eq:S,r}, we obtain
\begin{equation}
    \sqrt{n}\left(\frac{1}{n-2} \sum_{j=3}^{n}|\hat{\ep}_{C,j}|^{r}- \E_{\tz}[|\ep_{C,3}|^r]\right) = O_p(1).
    \nn
\end{equation}
By this estimate and Cram\'{e}r's theorem \rev{\cite[p.45]{Fer96}}, it suffices for \eqref{eq:C,<r>} to prove
\begin{equation}
\sqrt{n}\left|\frac{1}{n-2} \sum_{j=3}^{n}|\hat{\ep}_ {C,j}|^{\la r \ra}- \E_{\tz}[|\ep_{C,3}|^{\la r \ra}] \right|=O_p(1). \label{<r>:tight}
\end{equation}
We note that for all $x,y \in \mbbr$ and $r \in (0,1)$,
\begin{align}
||x|^{\la r \ra}- |y|^{\la r \ra}|
&=||x|^{\la r \ra}- |y|^{\la r \ra}|I(\sgn(x)=\sgn(y)) 
\nn\\
&{}\qquad + ||x|^{\la r \ra}- |y|^{\la r \ra}|I(\sgn(x) \neq \sgn(y))\\
&\leq ||x|^{r}- |y|^{r}| + \left\{|x|^{r} + |y|^{r}\right\}I(\sgn(x) \neq \sgn(y))\\
&\leq ||x|^{r}- |y|^{r}| + \left\{|x|^{r} + |y|^{r}\right\}I(|x-y| > |y|)
\\
&\lesssim |x- y|^{r}.
\end{align}
It follows that
\begin{align}
& \sqrt{n} \Biggl|\frac{1}{n-2} \sum_{j=3}^{n}|\hat{\ep}_ {C,j}|^{\la r \ra}- \E_{\tz}[|\ep_{C,3}|^{\la r \ra}]\Biggr|
\nn\\
&\leq \sqrt{n} \Biggl|\frac{1}{n-2} \sum_{j=3}^{n} \left(|\ep_ {C,j}|^{\la r \ra}- \E_{\tz}[|\ep_{C,3}|^{\la r \ra}]\right) \Biggr|
+ \frac{\sqrt{n}}{n-2} \sum_{j=3}^{n}\left||\hat{\ep}_{C,j}|^{\la r \ra}- |\ep_{C,j}|^{\la r \ra}\right|\\
&\lesssim O_p(1) + \frac{\sqrt{n}}{n-2} \sum_{j=3}^{n}|\hat{\ep}_{C,j}- \ep_{C,j}|^{r} = O_p(1),
\end{align}
hence \eqref{<r>:tight}.
\end{proof}

\medskip

Now we introduce our estimator $(\hat{\al}_n^{(0)}, \hat{\be}_n^{(0)}, \hat{\sig}_n^{(0)})$ of $(\al,\beta,\sig)$.
From \cite{Kur01},
\begin{equation}\label{hm:S-moment}
\E_{\theta} [|\ep_{S,2}|^r]= \sig^r\,\frac{\Gamma(1- \frac{r}{\al})}{\Gamma(1-r)} \frac{2^{\frac{r}{\al}}}{\cos \frac{r \pi}{2}}.
\end{equation}
Let $h_{r}(\al,\sig):= \E_{\tz} [|\ep_{S,2}|^r]$.
The moment estimator $\hat{\al}_n^{(0)}$ is the solution of the following equation:
\begin{equation}
\frac{(\frac{1}{n-1} \sum_{j=2}^{n}|\hat{\ep}_ {S,j}|^{r})^2}{\frac{1}{n-1} \sum_{j=2}^{n}|\hat{\ep}_ {S,j}|^{2r}} = \frac{h_{r}(\al,\sig)}{h_{2r}(\al,\sig)}.
\end{equation}
Since $\sig$ on the right hand side disappears, we can obtain $\hat{\al}_n^{(0)}$ by solving this equation. Then, we can obtain $\hat{\sig}_n^{(0)}$ by solving
\begin{equation}
\frac{1}{n-1} \sum_{j=2}^{n}|\hat{\ep}_ {S,j}|^{r} = h_{r}(\hat{\al}_n^{(0)},\sig).
\end{equation}
Next, to obtain the moment estimator of $\beta$, we note the following expressions \cite{Kur01}:
\begin{equation}
\E_{\theta}[|\ep_{C,3}|^{r}]= \frac{\Gamma(1- \frac{r}{\al})}{\Gamma(1-r)} \left| \frac{(2+2^{\al})\sig^{\al}}{\cos\eta}\right|^{\frac{r}{\al}} \frac{\cos\left( \frac{r\eta}{\al}\right)}{\cos\left( \frac{r\pi}{2}\right)}, 
\end{equation}
\begin{equation}
\E_{\theta}[|\ep_{C,3}|^{\la r \ra}]= \frac{\Gamma(1- \frac{r}{\al})}{\Gamma(1-r)} \left| \frac{(2+2^{\al})\sig^{\al}}{\cos\eta}\right|^{\frac{r}{\al}} \frac{\sin\left( \frac{r\eta}{\al}\right)}{\sin\left( \frac{r\pi}{2}\right)},
\end{equation}
where
\begin{equation}
    \eta:=\arctan\left( \left(\frac{2-2^\al}{2+2^{\al}} \right) \be \tan\frac{\al \pi}{2}\right).
\end{equation}
Taking the ratio of the above two equations together with their empirical counterparts and then substituting $\hat{\al}_n^{(0)}$ into $\al$, we obtain the following estimating equation for $\eta$:
\begin{equation}
\frac{\frac{1}{n-2} \sum_{j=3}^{n}|\hat{\ep}_ {C,j}|^{\la r \ra}}{ \frac{1}{n-2}\sum_{j=3}^{n}|\hat{\ep}_ {C,j}|^{r}} 
= \frac{\tan\frac{r\eta}{\hat{\al}_n^{(0)}}}{\tan{\frac{r\pi}{2}}}.
\end{equation}
By solving the above equation,
\begin{equation}
\hat{\eta}_n = \frac{\hat{\al}_n^{(0)}}{r} \arctan \left( \tan\left( \frac{r\pi}{2}\right) \frac{\frac{1}{n-2} \sum_{j=3}^{n}|\hat{\ep}_ {C,j}|^{\la r \ra}}{ \frac{1}{n-2}\sum_{j=3}^{n}|\hat{\ep}_ {C,j}|^{r}} \right).
\end{equation}
From this and the definition of $\eta$,
\begin{equation}
\hat{\be}_n^{(0)} = \left( \frac{2+2^{\hat{\al}_n ^{(0)}}}{2-2^{\hat{\al}_n^{(0)}}}\right)
\frac{\tan\hat{\eta}_n}{\tan{\frac{\hat{\al}_n \pi}{2}}}.
\end{equation}
Let $\tes^{(0)}:= (\hat{\al}_n^{(0)}, \hat{\be}_n^{(0)}, \hat{\sig}_n^{(0)}, \hat{\mu}_n)$ with $\hat{\mu}_n$ being the CQMLE.
Our main claim is the following.

\begin{thm} \label{init_tight}
Under Assumption \ref{assump_AN_TPE},
\begin{equation}
\sqrt{n}\left(\tes^{(0)}-\tz \right)=O_p(1).
\end{equation}
\end{thm}

\begin{proof}
By Theorem \ref{hm:thm_CQMLE}, it suffices to show the tightness $\sqrt{n}(\aes^{(0)}-\al_0,\,\ses^{(0)}-\sig_0)$ and $\sqrt{n}(\hat{\beta}_{n}^{(0)}- \beta_{0})$ separately.

Let $m_1(\al,\sig):=\left(\E_{\theta}[|\ep_{S,2}|^{r}],\,\E_{\theta}[|\ep_{S,2}|^{2r}]\right)$ and $M(\al,\sig):= \p_{(\alpha,\sig)}m_{1}(\al, \sig)$.
By the expression \eqref{hm:S-moment}, the determinant of $M(\az,\sz)$ is written as
\begin{align}
|M(\az,\sz)|
&=\frac{\Gam(1-\frac{r}{\az}) \Gam(1-\frac{2r}{\az})}{\Gam(1-2r)\Gam(1-r)} \frac{2^{\frac{3r}{\az}+1}\sz^{3r+1}r^2}{\az^2 \cos(\frac{r\pi}{2})\cos(r\pi)} \left(\psi \left(1-\frac{r}{\az}\right)- \psi \left(1-\frac{2r}{\az}\right)\right),
\end{align}
where $\psi$ denotes the digamma function.
Since
\begin{equation}
    \frac{\Gam(1-\frac{r}{\az}) \Gam(1-\frac{2r}{\az})}{\Gam(1-2r)\Gam(1-r)} \frac{2^{\frac{3r}{\al}+1}\sig^{3r+1}r^2}{\al^2 \cos(\frac{r\pi}{2})\cos(r\pi)} \neq 0
\end{equation}
under \eqref{hm:r_region} and since $\psi(1-\frac{r}{\az})- \psi(1-\frac{2r}{\az}) \neq 0$ because $\psi$ is strictly increasing, it follows that $M(\az,\sz)$ is non-singular.
Lemma \ref{moment_tight} and Cram\'{e}r's theorem give the tightness of $\sqrt{n}(\aes^{(0)}-\al_0,\,\ses^{(0)}-\sig_0)$.

Write $\bes^{(0)}= H_n(\aes^{(0)})$, where
\begin{align}
H_{n}(\al) &:= 
\frac{(2+2^{\al})\tan(\frac{\al}{r}\arctan(T_n \tan(\frac{r\pi}{2})))}{(2-2^{\al})\tan(\frac{\al\pi}{2})},\label{hm:H_def}\\
T_n &:=\frac{1}{n-2} \sum_{j=3}^{n}|\hat{\ep}_ {C, j}|^{\la r \ra} \bigg/{\frac{1}{n-2}\sum_{j=3}^{n}|\hat{\ep}_ {C, j}|^{r}}.
\end{align}
The tightness $\sqrt{n}(\bes^{(0)} - \beta_{0})=\sqrt{n}\{ H_{n}(\aes^{(0)})- \beta_{0}\}=O_p(1)$ follows from showing $\sqrt{n} \{H_{n}(\al_0)- \beta_{0}\} =O_p(1)$ and $\p_{\al}H_{n}(\tilde{\al}_{n})=O_p(1)$.
The former is obvious by Cram\'{e}r's theorem and the latter follows from Lemma \ref{hm:lem-add1} below.
The proof is complete.
\end{proof}

\rev{
It is possible to deduce the asymptotic normality of $\tes^{(0)}$ at rate $\sqrt{n}$ by the standard theory based on the delta method (see \cite[Section 7]{Fer96}, \cite[Chapter 4]{vdV98}, and so on). Nevertheless, the derivation of the closed form of the asymptotic covariance matrix would be very complicated and messy, and we will not pursue it.
}

\begin{lem}
\label{hm:lem-add1}
For any $\tilde{\al}_{n} \cip \alpha_0$, we have
\begin{equation}
\p_{\al}H_{n}(\tilde{\al}_{n}) \cip \frac{\p_{\al} g(T_0, \al_{0})h(\al_{0})-g(T_0, \al_{0})\p_{\al}h(\al_{0})}{h(\al_{0})^2},
\end{equation}
where $h(\al)$ and $g(x,\al)$ are defined in the proof.
 \end{lem}

\begin{proof}
Denote the numerator and denominator in \eqref{hm:H_def} by $g(x,\al)$ and $h(\al)$, respectively.
It is obvious that $(x,\al) \mapsto g(x,\al)$ is $C^1$-class. 
%
It remains to show that $h(\al) \neq 0$ for $\al \in (0,2)$ and that $h(\al)$ is $C^1$-class.
It is obvious that $h(\al) > 0$; as was mentioned in the introduction, we can easily check $\lim_{\al \rightarrow 1}h(\al)=\frac{4}{\pi}\log2$ \cite[p.189]{Nol98}.
For $\al\ne 1$,
\begin{align}
\p_{\al}h(\al)
&=-(\log2)2^{\al}\tan \left(\frac{\al\pi}{2} \right) + \frac{\pi}{2}\frac{(2-2^{\al})}{\cos^2(\frac{\al\pi}{2})}\\
&= \frac{\pi}{2}(2-2^{\al})\tan^2 \left(\frac{\al\pi}{2} \right)- (2\log2)\tan \left(\frac{\al\pi}{2}\right)+\frac{\pi}{2}(2-2^{\al})
\nn\\
&{}\qquad +(\log2)(2-2^{\al})\tan \left(\frac{\al\pi}{2} \right).
\end{align}
Again by the fact that $\lim_{\alpha \to 1} \tan(\al\pi/2) (|u|^{1 - \alpha} - 1) = (2/\pi)\log |u|$, it is straightforward to see that $\p_{\al}h(\al)$ is continuous in $\al\in(0,2)$.
\end{proof}

\begin{rem}
\label{hm:rem_onestep}
Let $\ell_n(\theta)$ denote the log-likelihood function based on $\{(X_j,Y_j)\}_{j\leq n}$:
\begin{equation}
    \ell_n(\theta) = \sumj \log\left\{
    \frac{1}{\sigma}\phi_{\al,\be}\left(
    \frac{Y_j - X_j \cdot \mu}{\sigma}
    \right)\right\}.
\end{equation}
It is known that the one-step estimator $\tes^{(1)} := \tes^{(0)} - {\ell_n'(\tes^{(0)})}^{-1}l_n''(\tes^{(0)})$ is asymptotically equivalent to the MLE when $\ell_n$ is smooth enough and $\tes^{(0)}$ is $\sqrt{n}$-consistent \cite[Theorem 5.5.4]{Zac71}. 
However, the calculation of $\tes^{(1)}$ turned out to be very time-consuming and unstable.
The direct numerical optimization of $\ell_n$ with reasonable initial values may result in better performance; see Table \ref{tab:numerical} in the introduction.
We remark that, in the i.i.d.-$S^{0}_{\alpha} (\beta,\sigma,\mu)$ framework, \cite{Mat21} recently studied the fundamental properties and asymptotic behaviors of the likelihood functions and the associated MLE.
\end{rem}

\begin{rem}[Incorporation of intercept]
    \label{hm:rem_location}
Suppose that none of the $q$ components of $j\mapsto X_j$ is constant and that the model is
\begin{equation}
    Y_j=m+X_j\cdot\mu + \ep_j
\end{equation}
with an additional unknown location parameter $m\in\mbbr$. Under suitable conditions on $X_1,X_2,\dots$ as before, 
we can estimate $\theta=(\al,\beta,\sig,\mu)$ without reference to $m$; 
the intercept $m$ is not estimable by the procedure we have seen so far since taking differences eliminates $m$.
Nevertheless, once we obtain $\tes^{(0)}=(\aes^{(0)},\bes^{(0)},\ses^{(0)},\mes^{(0)})$, 
\rev{
loosely speaking, the random variables $(\ses^{(0)})^{-1}(Y_j - m_0 - X_j\cdot \mes^{(0)})$ ($j\le n$) are approximately distributed as $S_{\aes^{(0)}}(\bes^{(0)},1,0)$, where $m_0$ denotes the true value of $m$.
}
According to the standard $M$-estimation theory, it is readily expected that we could estimate $m$ at rate $\sqrt{n}$ by any element maximizing the random function
\begin{align*}
    m &\in \argmax_m 
    \sumj 
    \log\left(
    \frac{1}{\ses^{(0)}}\phi_{\aes^{(0)},\bes^{(0)}}
    \left(\frac{Y_j - m - X_j\cdot \mes^{(0)}}{\ses^{(0)}}\right)
    \right)
    \nn\\
    &=\argmax_m 
    \sumj 
    \log\phi_{\aes^{(0)},\bes^{(0)}}
    \left(\frac{Y_j - m - X_j\cdot \mes^{(0)}}{\ses^{(0)}}\right).
\end{align*}
\end{rem}



\section{Numerical experiments}\label{sec:numerical}

In this section, we conduct the numerical experiments of the parameter estimation. We consider the following statistical model:
\begin{equation}
Y_j = X_j \cdot \mu + \sigma \epsilon_j,
\end{equation}
where $X_j,\mu \in \mbbr^3$. The simulation design is as follows.
\begin{itemize}
\item True value:\ $\al\in\{0.8, 1.0,1.5\}$,\ $\be=0.5$,\ $\sig=1.5$,\ $\mu=(5,2,3)$.
\item Sample size:\ $n=300,~500,~1000,~1500$.
\item Each component of $X_j$ is generated independently from the uniform distribution $U(1,5)$.
\end{itemize}
\rev{We set $r=0.01$ for the order of the moment estimator.}
We generated $\ep_1,\ep_2,\dots$ by the method based on \cite[Theorem 1.3]{Nol20} and used the R package $\mathtt{stabledist}$ to calculate the probability density function.
For each setting, we implemented 1000 Monte Carlo trials.

In Section \ref{hm:sec_sim.IE}, we will observe finite-sample performance of the initial estimator studied in Sections \ref{sec:CQMLE} and \ref{sec:stable}. For comparison, we also compute 
the LADE $\hat{\mu}_{LADE}$ and the least-squares estimator (LSE) $\hat{\mu}_{LSE}$, and also the estimators of $(\al,\beta,\sigma)$ through the method of moments based on the $\hat{\mu}_{LADE}$ and $\hat{\mu}_{LSE}$. 
Here, the LADE and LSE are defined as any elements such that
\begin{align}
\hat{\mu}_{LADE} &\in \argmin_{\mu \in \Theta_{\mu}} \sum_{j=2}^{n}|\Delta_jY - \Delta_jX \cdot \mu|,
\nn\\
\hat{\mu}_{LSE} &\in \argmin_{\mu \in \Theta_{\mu}} \sum_{j=2}^{n}|\Delta_jY - \Delta_jX \cdot \mu|^2;
\end{align}
The LSE's asymptotic (non-normal) distribution is known under some constraints, but practically inconvenient to handle.
Then, in Section \ref{hm:sec_sim.MLE}, we will present the histograms of the MLE of $\theta$ through the numerical optimization with the CQMLE and the associated initial estimator of $(\al,\beta,\sigma)$ as the initial value for numerical search at each trial.

\subsection{Initial estimators}
\label{hm:sec_sim.IE}


For each true value, we first drew the boxplots of the initial estimators when we used the LADE, the LSE, and the CQMLE as the estimator of the regression coefficient.
Next, we drew the histograms of the normalized estimator $\sqrt{n}(\tes^{(0)} - \tz)$ for $n=300,~500$.

The boxplots in Figures \ref{hm:fig_bp-1}, \ref{hm:fig_bp-2} and \ref{hm:fig_bp-3} show that the initial estimators when we use the LADE or the CQMLE converge to the true value as the sample size increases. Also, the initial estimator when we use the CQMLE converges faster than that with the LADE. When we use the LSE, the initial estimator does not converge to the true value. It is proved theoretically that the initial estimator when we use the LSE does not converge to the true value when $0<\alpha_0\le 1$. 
This phenomenon could also be observed in the corresponding histograms (Figures \ref{hm:fig_hist-08-1}, \ref{hm:fig_hist-08-2}, \ref{hm:fig_hist-1-1}, \ref{hm:fig_hist-1-2}, 
\ref{hm:fig_hist-15-1}, and \ref{hm:fig_hist-15-2}).

Overall, the proposed estimators: the CQMLE and the subsequent moment estimators are numerically stable and require short computation time, so they are suitable for the initial values of the maximum likelihood estimation.


\begin{figure}[h]
  \begin{minipage}[b]{0.45\linewidth}
    \centering
    \includegraphics[width=70mm, height=60mm]{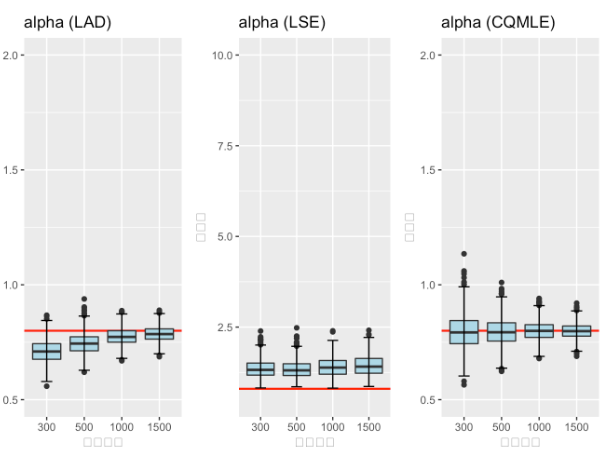}
  \end{minipage}
  \begin{minipage}[b]{0.45\linewidth}
    \centering
    \includegraphics[width=70mm, height=60mm]{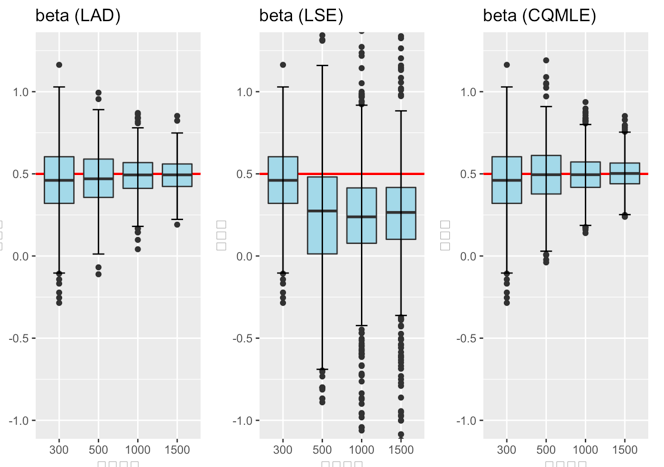}
  \end{minipage}
  \begin{minipage}[b]{0.45\linewidth}
    \centering
    \includegraphics[width=70mm, height=60mm]{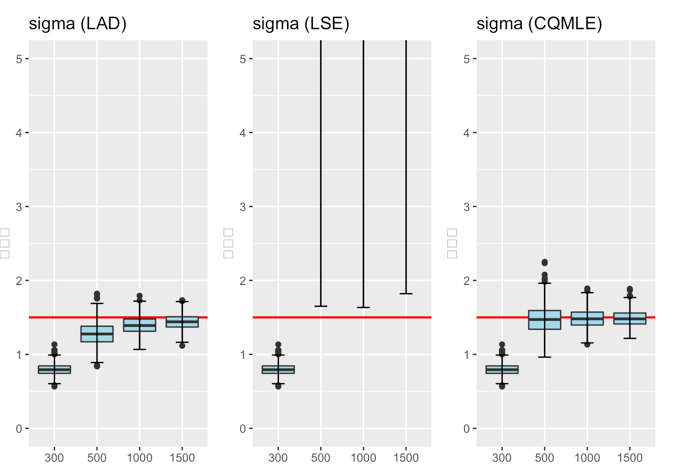}
  \end{minipage}
  \begin{minipage}[b]{0.45\linewidth}
    \centering
    \includegraphics[width=70mm, height=60mm]{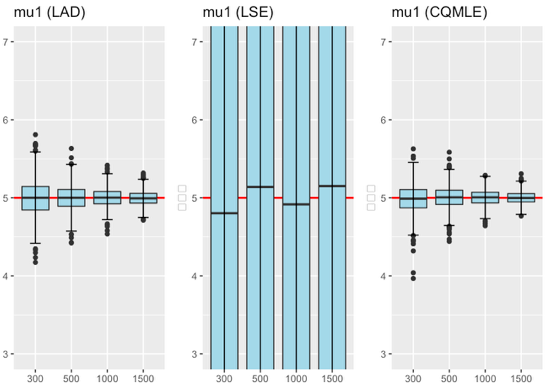}
  \end{minipage}
  \begin{minipage}[b]{0.45\linewidth}
    \centering
    \includegraphics[width=70mm, height=60mm]{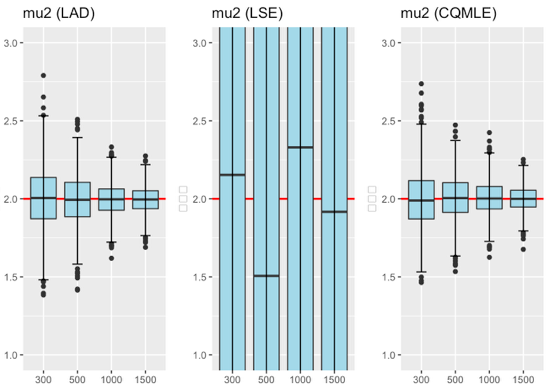}
  \end{minipage}
  \begin{minipage}[b]{0.45\linewidth}
    \centering
    \includegraphics[width=70mm, height=60mm]{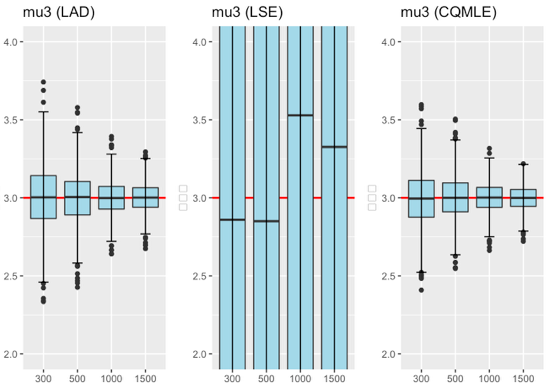}
  \end{minipage}
  \caption{Boxplots of the initial estimators for $n=300,~500,~1000,~1500$ when the true value is $(\al, \be, \sig, \mu)= (0.8, 0.5, 1.5, (5,2,3))$; in each panel, the red line shows the true value.
      }  
      \label{hm:fig_bp-1}
\end{figure}
\newpage
\begin{figure}[h]
  \begin{minipage}[b]{0.45\linewidth}
    \centering
    \includegraphics[width=70mm, height=60mm]{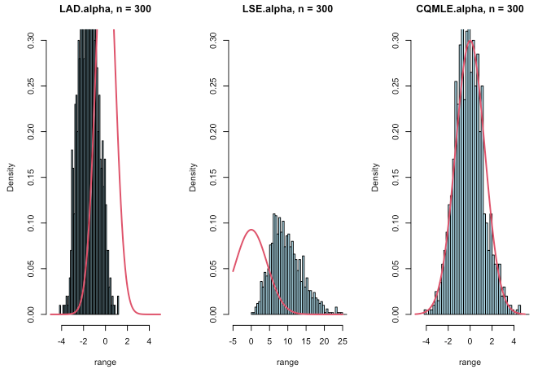}
  \end{minipage}
  \begin{minipage}[b]{0.45\linewidth}
    \centering
    \includegraphics[width=70mm, height=60mm]{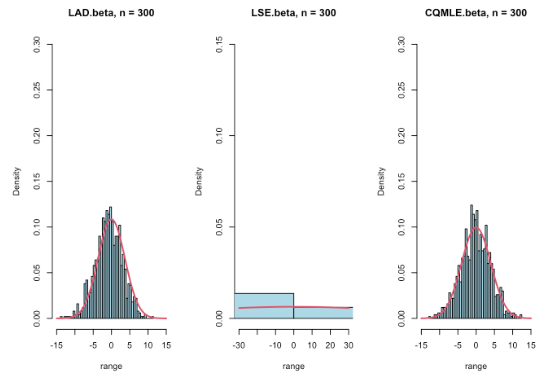}
  \end{minipage}
  \begin{minipage}[b]{0.45\linewidth}
    \centering
    \includegraphics[width=70mm, height=60mm]{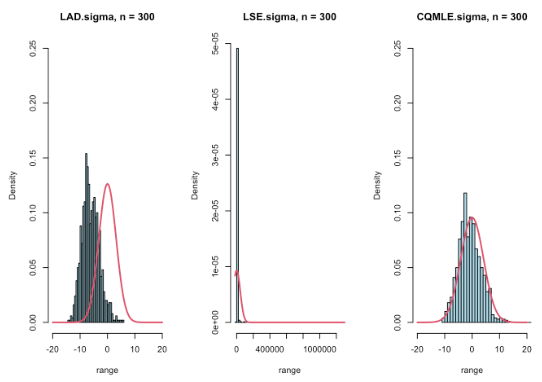}
  \end{minipage}
  \begin{minipage}[b]{0.45\linewidth}
    \centering
    \includegraphics[width=70mm, height=60mm]{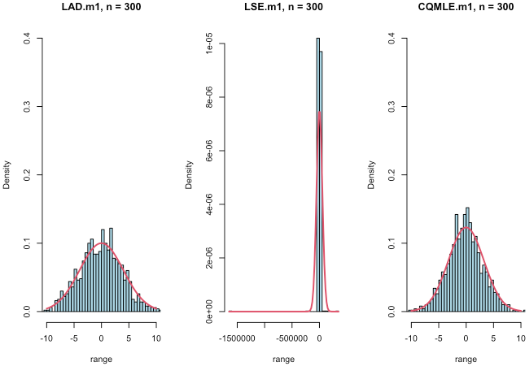}
  \end{minipage}
  \begin{minipage}[b]{0.45\linewidth}
    \centering
    \includegraphics[width=70mm, height=60mm]{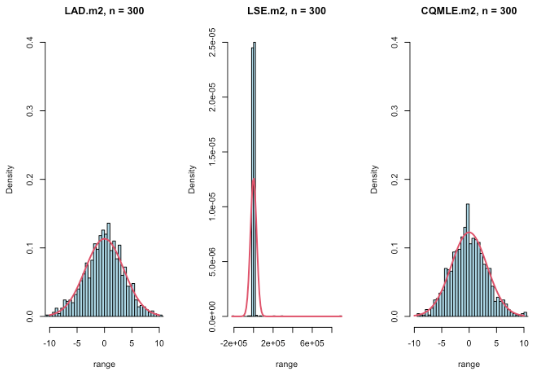}
  \end{minipage}
  \begin{minipage}[b]{0.45\linewidth}
    \centering
    \includegraphics[width=70mm, height=60mm]{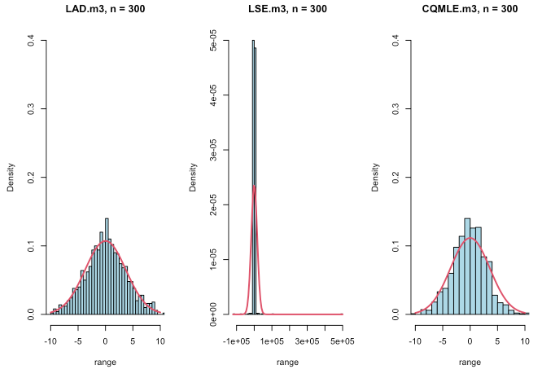}
  \end{minipage}
  \caption{Histograms of the initial estimators for $n=300$ when the true value is $(\al, \be, \sig, \mu)= (0.8, 0.5, 1.5, (5,2,3))$; in each panel, the red line represents the probability density function of the normal distribution with the mean 0 and the variance equal to that of the normalized estimator.
      }  
      \label{hm:fig_hist-08-1}
\end{figure}
\newpage
\begin{figure}[h]
  \begin{minipage}[b]{0.45\linewidth}
    \centering
    \includegraphics[width=70mm, height=60mm]{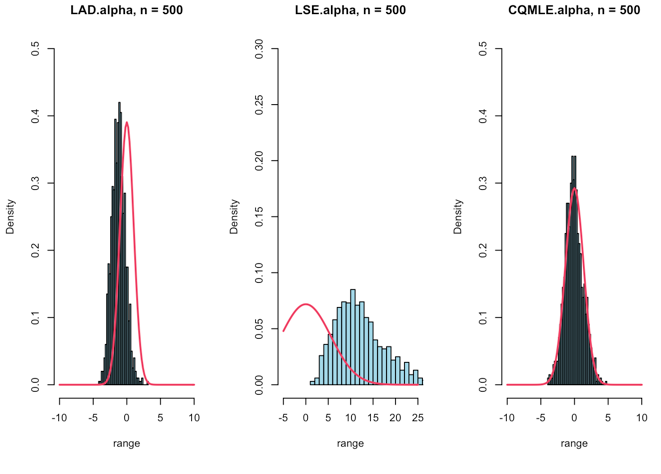}
  \end{minipage}
  \begin{minipage}[b]{0.45\linewidth}
    \centering
    \includegraphics[width=70mm, height=60mm]{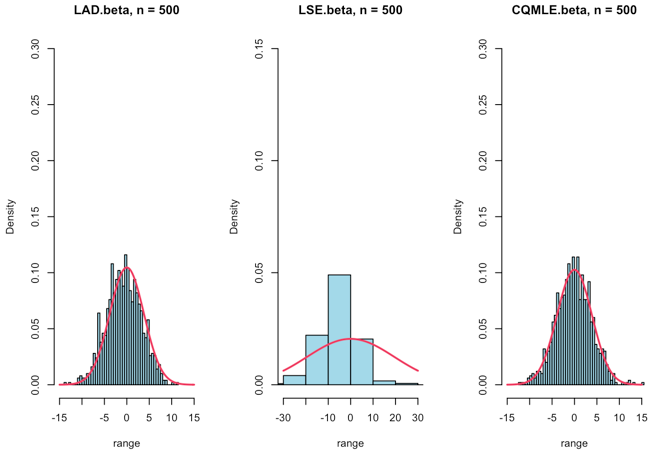}
  \end{minipage}
  \begin{minipage}[b]{0.45\linewidth}
    \centering
    \includegraphics[width=70mm, height=60mm]{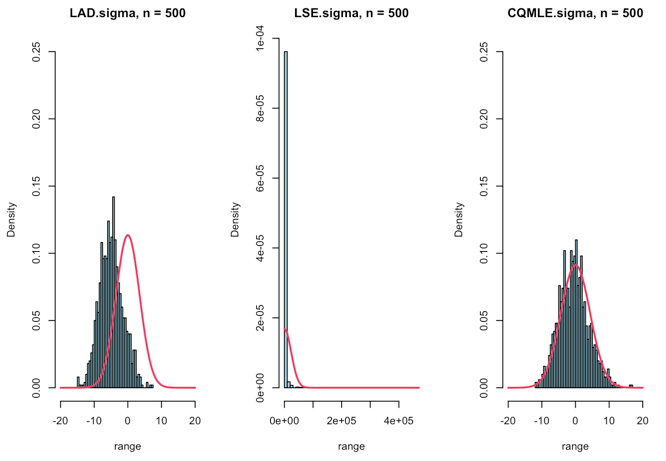}
  \end{minipage}
  \begin{minipage}[b]{0.45\linewidth}
    \centering
    \includegraphics[width=70mm, height=60mm]{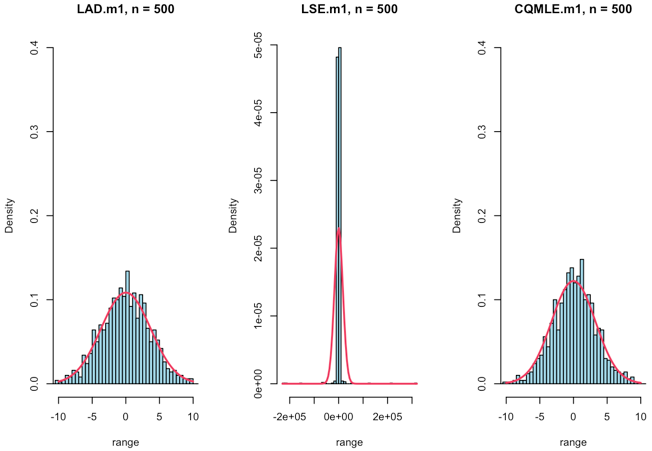}
  \end{minipage}
  \begin{minipage}[b]{0.45\linewidth}
    \centering
    \includegraphics[width=70mm, height=60mm]{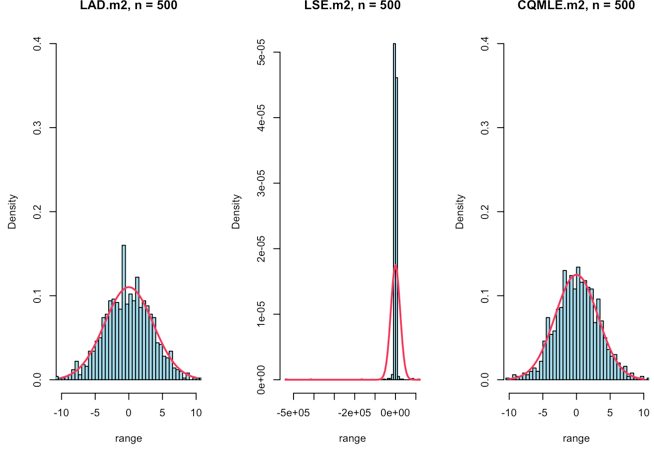}
  \end{minipage}
  \begin{minipage}[b]{0.45\linewidth}
    \centering
    \includegraphics[width=70mm, height=60mm]{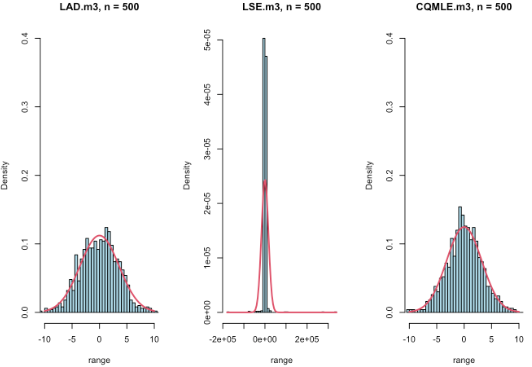}
  \end{minipage}
  \caption{Histograms of the initial estimators for $n=500$ when the true value is $(\al, \be, \sig, \mu)= (0.8, 0.5, 1.5, (5,2,3))$; in each panel, the red line represents the probability density function of the normal distribution with the mean 0 and the variance equal to that of the normalized estimator.
      }  
      \label{hm:fig_hist-08-2}
\end{figure}
\newpage
\begin{figure}[h]
  \begin{minipage}[b]{0.45\linewidth}
    \centering
    \includegraphics[width=70mm, height=60mm]{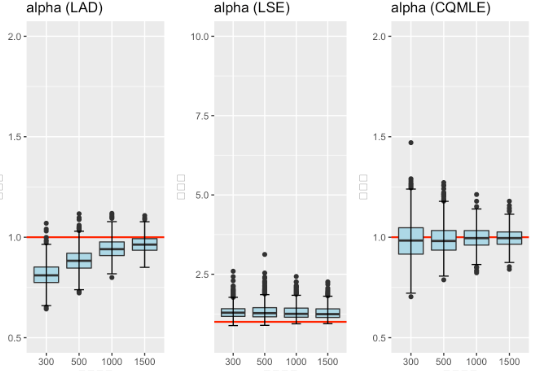}
  \end{minipage}
  \begin{minipage}[b]{0.45\linewidth}
    \centering
    \includegraphics[width=70mm, height=60mm]{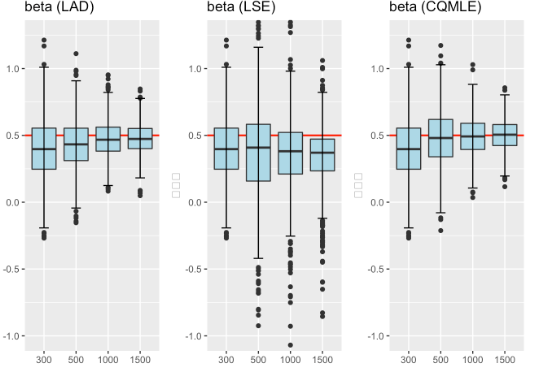}
  \end{minipage}
  \begin{minipage}[b]{0.45\linewidth}
    \centering
    \includegraphics[width=70mm, height=60mm]{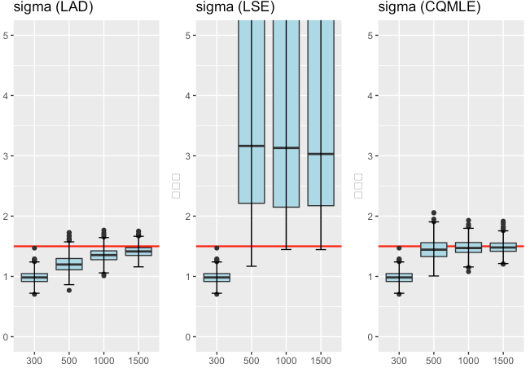}
  \end{minipage}
  \begin{minipage}[b]{0.45\linewidth}
    \centering
    \includegraphics[width=70mm, height=60mm]{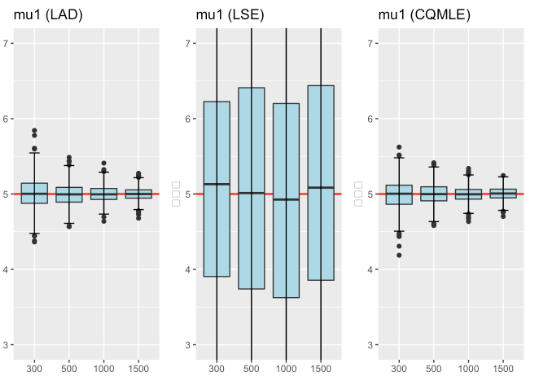}
  \end{minipage}
  \begin{minipage}[b]{0.45\linewidth}
    \centering
    \includegraphics[width=70mm, height=60mm]{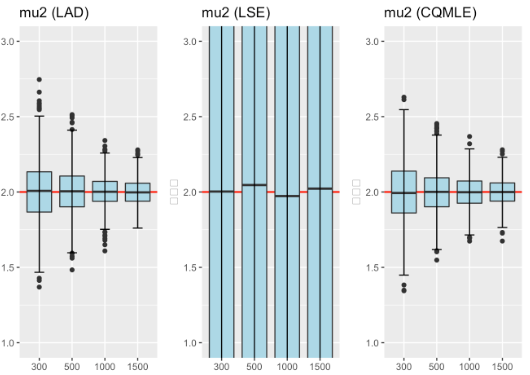}
  \end{minipage}
  \begin{minipage}[b]{0.45\linewidth}
    \centering
    \includegraphics[width=70mm, height=60mm]{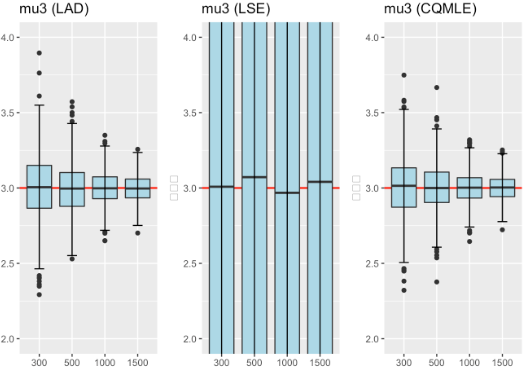}
  \end{minipage}
  \caption{Boxplots of the initial estimators for $n=300,~500,~1000,~1500$ when the true value is $(\al, \be, \sig, \mu)= (1.0, 0.5, 1.5, (5,2,3))$; in each panel, the red line shows the true value.
      }  
      \label{hm:fig_bp-2}
\end{figure}
\newpage
\begin{figure}[h]
  \begin{minipage}[b]{0.45\linewidth}
    \centering
    \includegraphics[width=60mm, height=55mm]{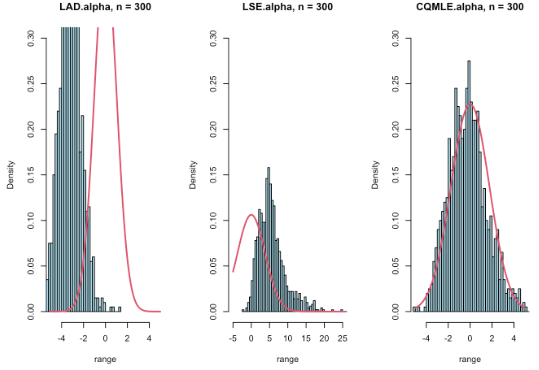}
  \end{minipage}
  \begin{minipage}[b]{0.45\linewidth}
    \centering
    \includegraphics[width=60mm, height=55mm]{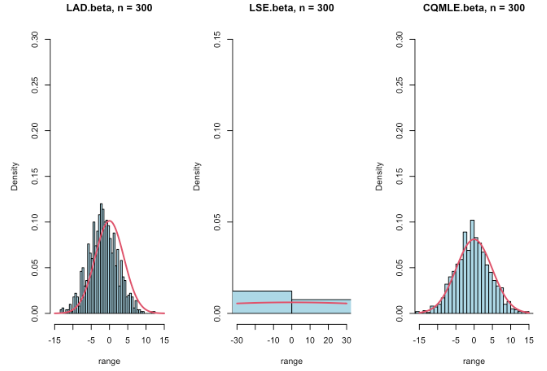}
  \end{minipage}
  \begin{minipage}[b]{0.45\linewidth}
    \centering
    \includegraphics[width=60mm, height=55mm]{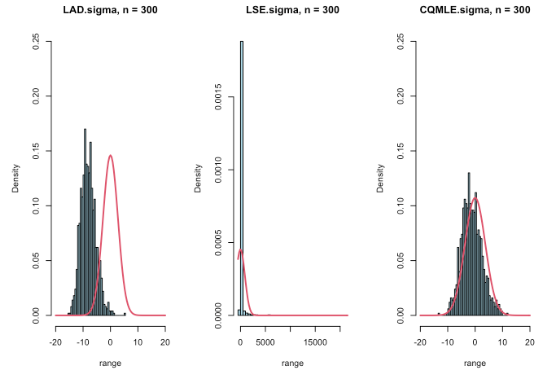}
  \end{minipage}
  \begin{minipage}[b]{0.45\linewidth}
    \centering
    \includegraphics[width=60mm, height=55mm]{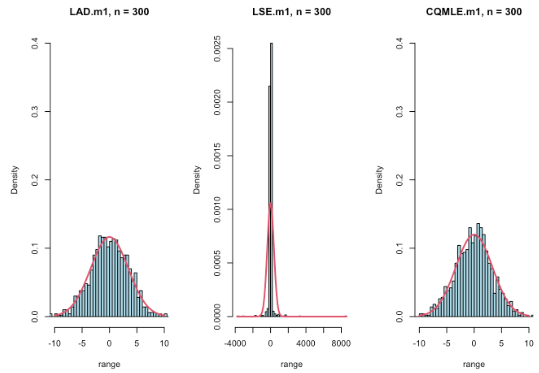}
  \end{minipage}
  \begin{minipage}[b]{0.45\linewidth}
    \centering
    \includegraphics[width=60mm, height=55mm]{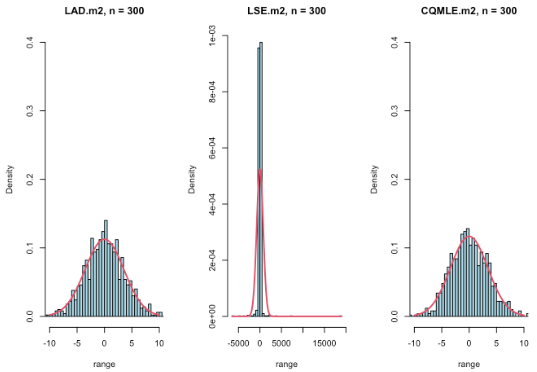}
  \end{minipage}
  \begin{minipage}[b]{0.45\linewidth}
    \centering
    \includegraphics[width=60mm, height=55mm]{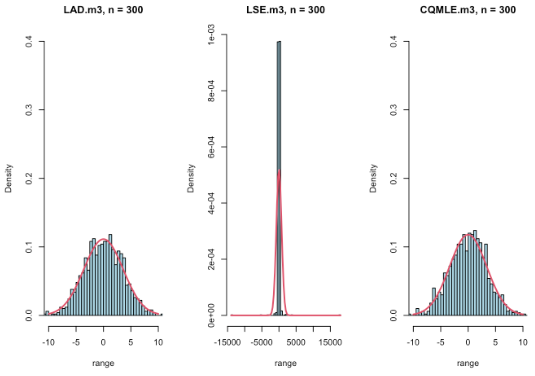}
  \end{minipage}
  \caption{Histograms of the initial estimators for $n=300$ when the true value is $(\al, \be, \sig, \mu)= (1.0, 0.5, 1.5, (5,2,3))$; in each panel, the red line represents the probability density function of the normal distribution with the mean 0 and the variance equal to that of the normalized estimator.
      }  
      \label{hm:fig_hist-1-1}
\end{figure}
\newpage
\begin{figure}[h]
  \begin{minipage}[b]{0.45\linewidth}
    \centering
    \includegraphics[width=60mm, height=60mm]{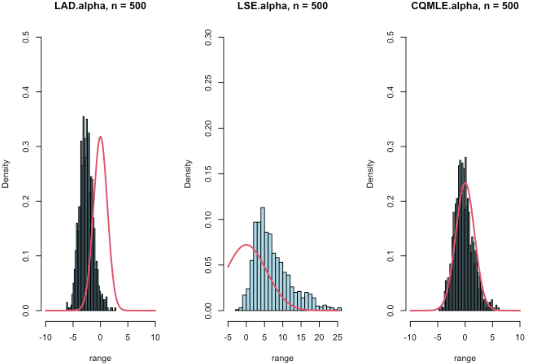}
  \end{minipage}
  \begin{minipage}[b]{0.45\linewidth}
    \centering
    \includegraphics[width=60mm, height=60mm]{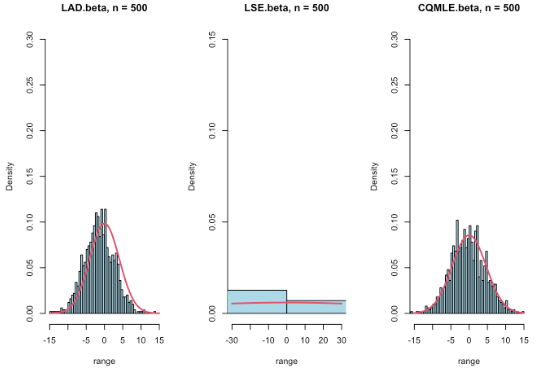}
  \end{minipage}
  \begin{minipage}[b]{0.45\linewidth}
    \centering
    \includegraphics[width=60mm, height=60mm]{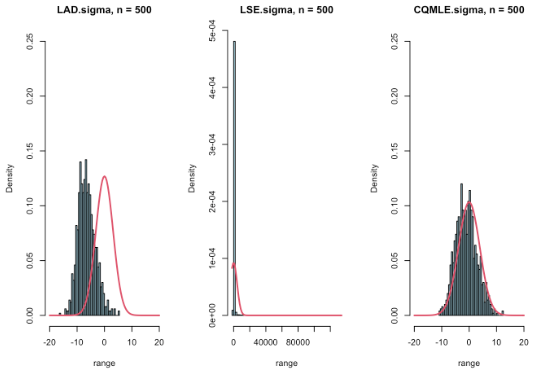}
  \end{minipage}
  \begin{minipage}[b]{0.45\linewidth}
    \centering
    \includegraphics[width=60mm, height=60mm]{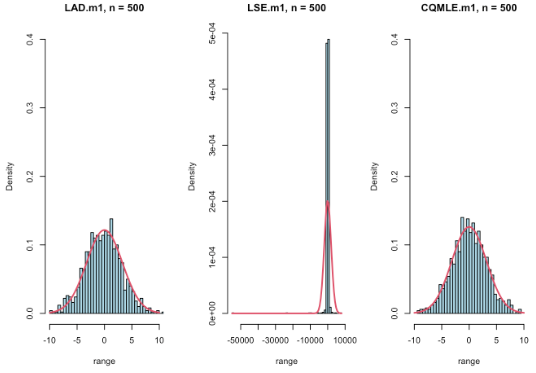}
  \end{minipage}
  \begin{minipage}[b]{0.45\linewidth}
    \centering
    \includegraphics[width=60mm, height=60mm]{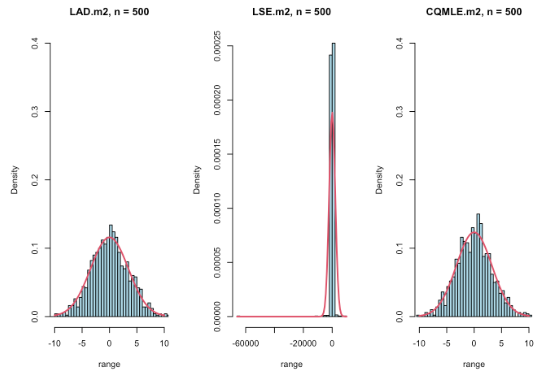}
  \end{minipage}
  \begin{minipage}[b]{0.45\linewidth}
    \centering
    \includegraphics[width=60mm, height=60mm]{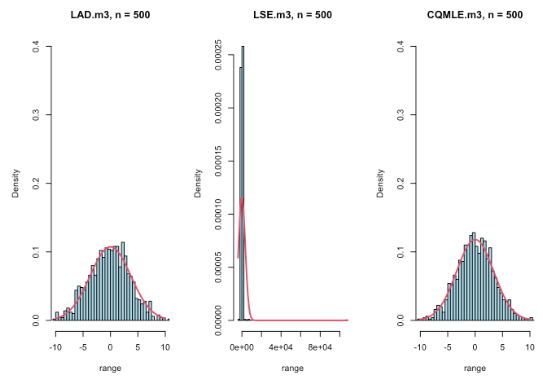}
  \end{minipage}
  \caption{Histograms of the initial estimators for $n=500$ when the true value is $(\al, \be, \sig, \mu)= (1.0, 0.5, 1.5, (5,2,3))$; in each panel, the red line represents the probability density function of the normal distribution with the mean 0 and the variance equal to that of the normalized estimator.
      }  
      \label{hm:fig_hist-1-2}
\end{figure}
\newpage
\begin{figure}[h]
  \begin{minipage}[b]{0.45\linewidth}
    \centering
    \includegraphics[width=70mm, height=60mm]{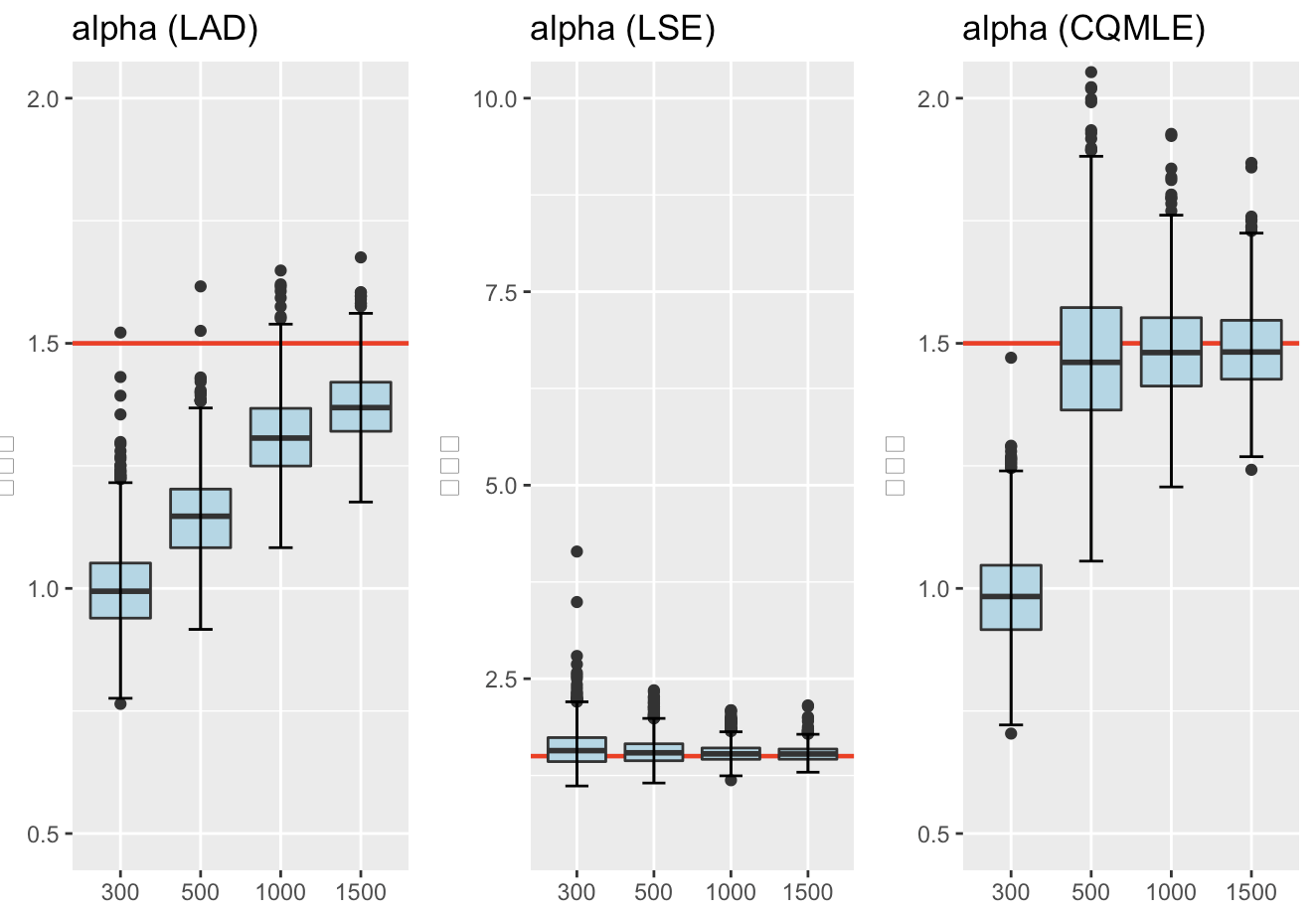}
  \end{minipage}
  \begin{minipage}[b]{0.45\linewidth}
    \centering
    \includegraphics[width=70mm, height=60mm]{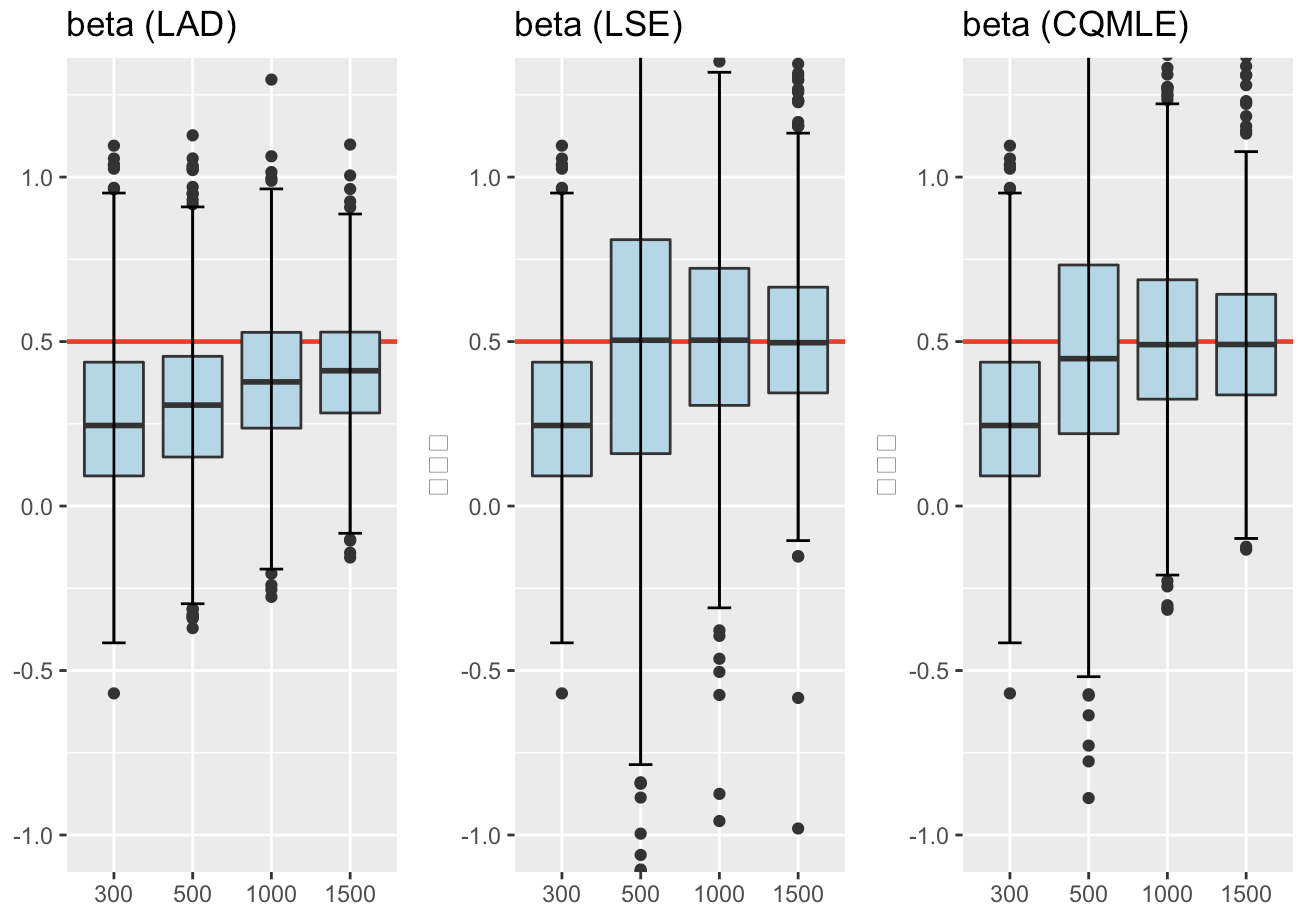}
  \end{minipage}
  \begin{minipage}[b]{0.45\linewidth}
    \centering
    \includegraphics[width=70mm, height=60mm]{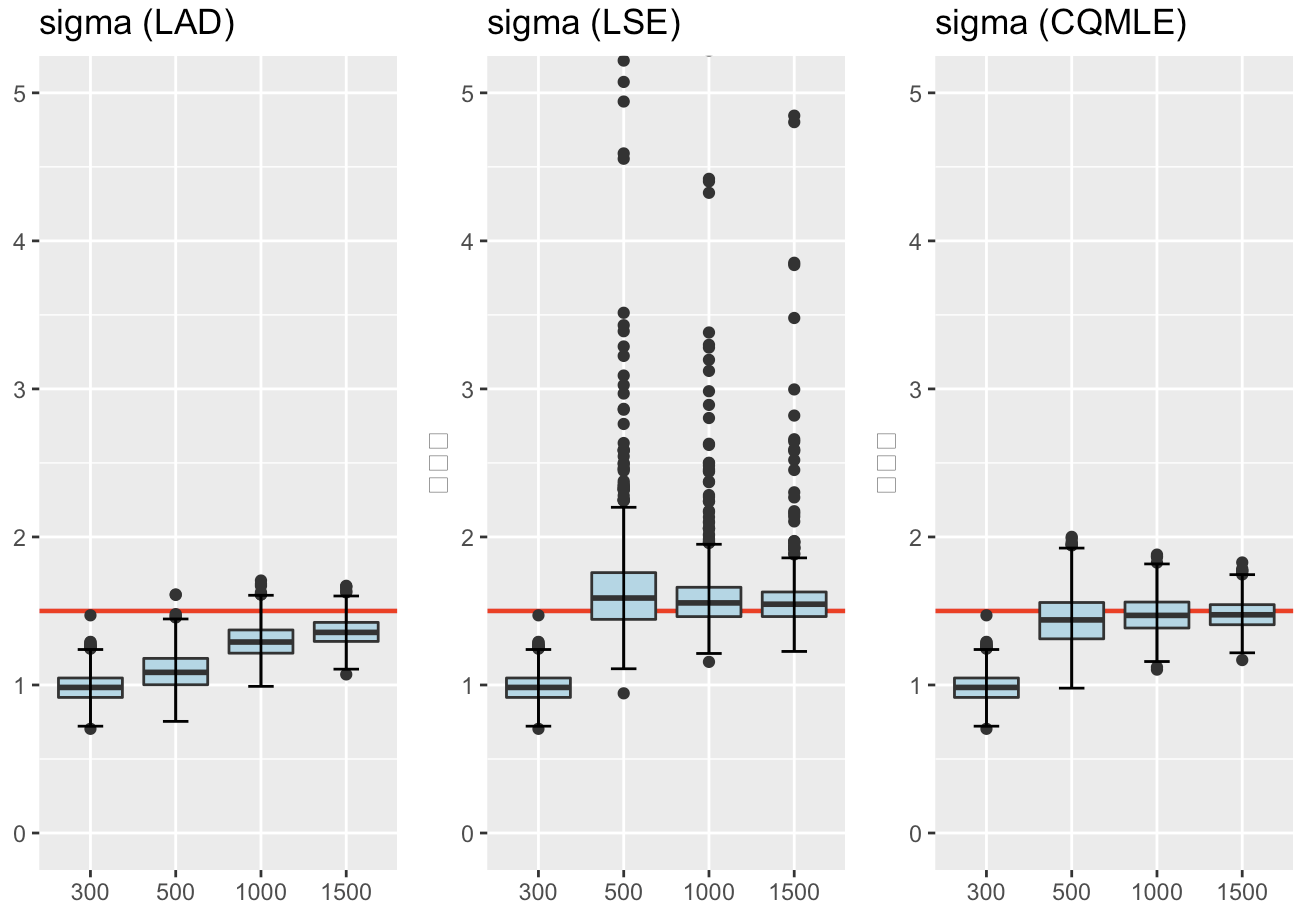}
  \end{minipage}
  \begin{minipage}[b]{0.45\linewidth}
    \centering
    \includegraphics[width=70mm, height=60mm]{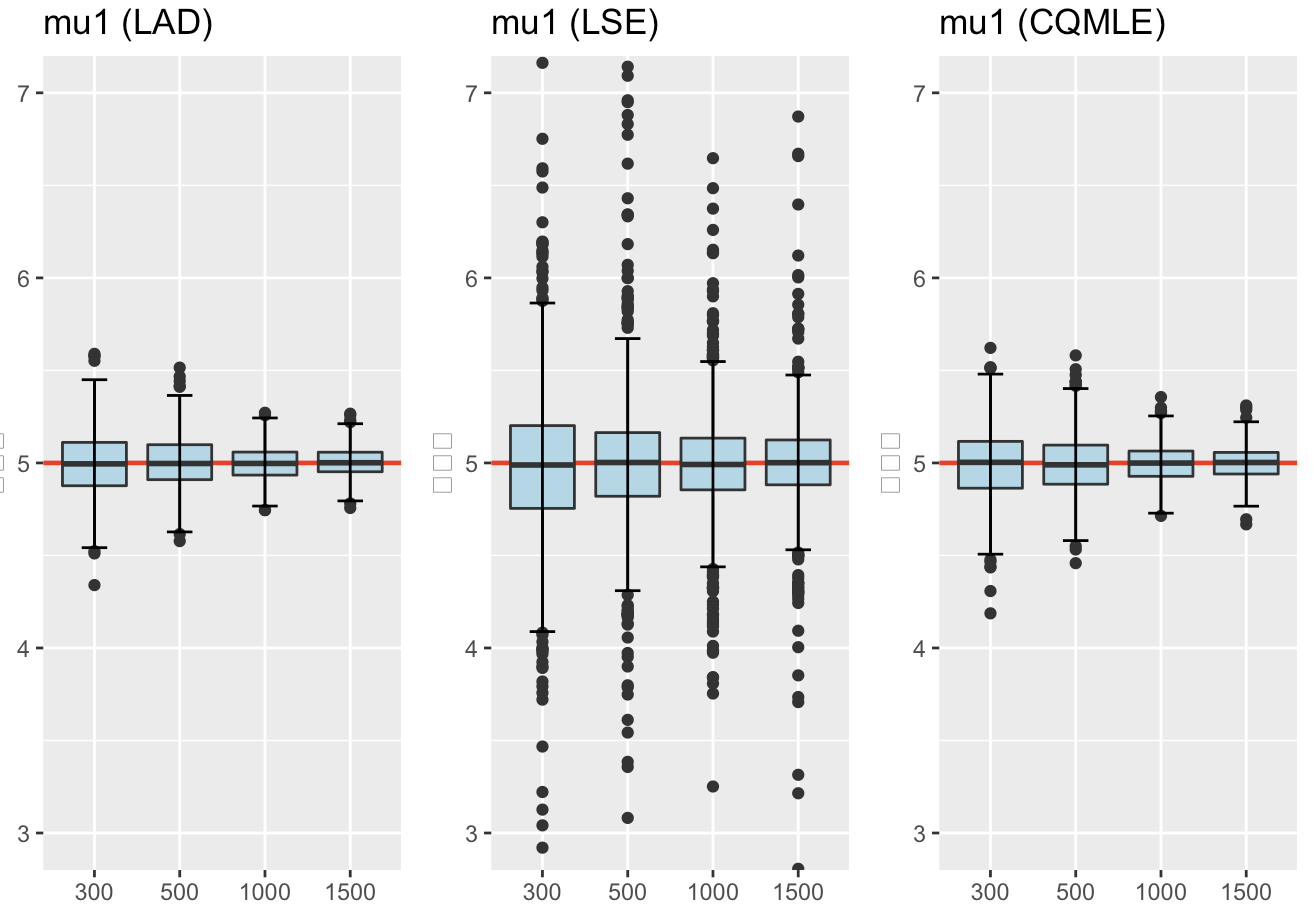}
  \end{minipage}
  \begin{minipage}[b]{0.45\linewidth}
    \centering
    \includegraphics[width=70mm, height=60mm]{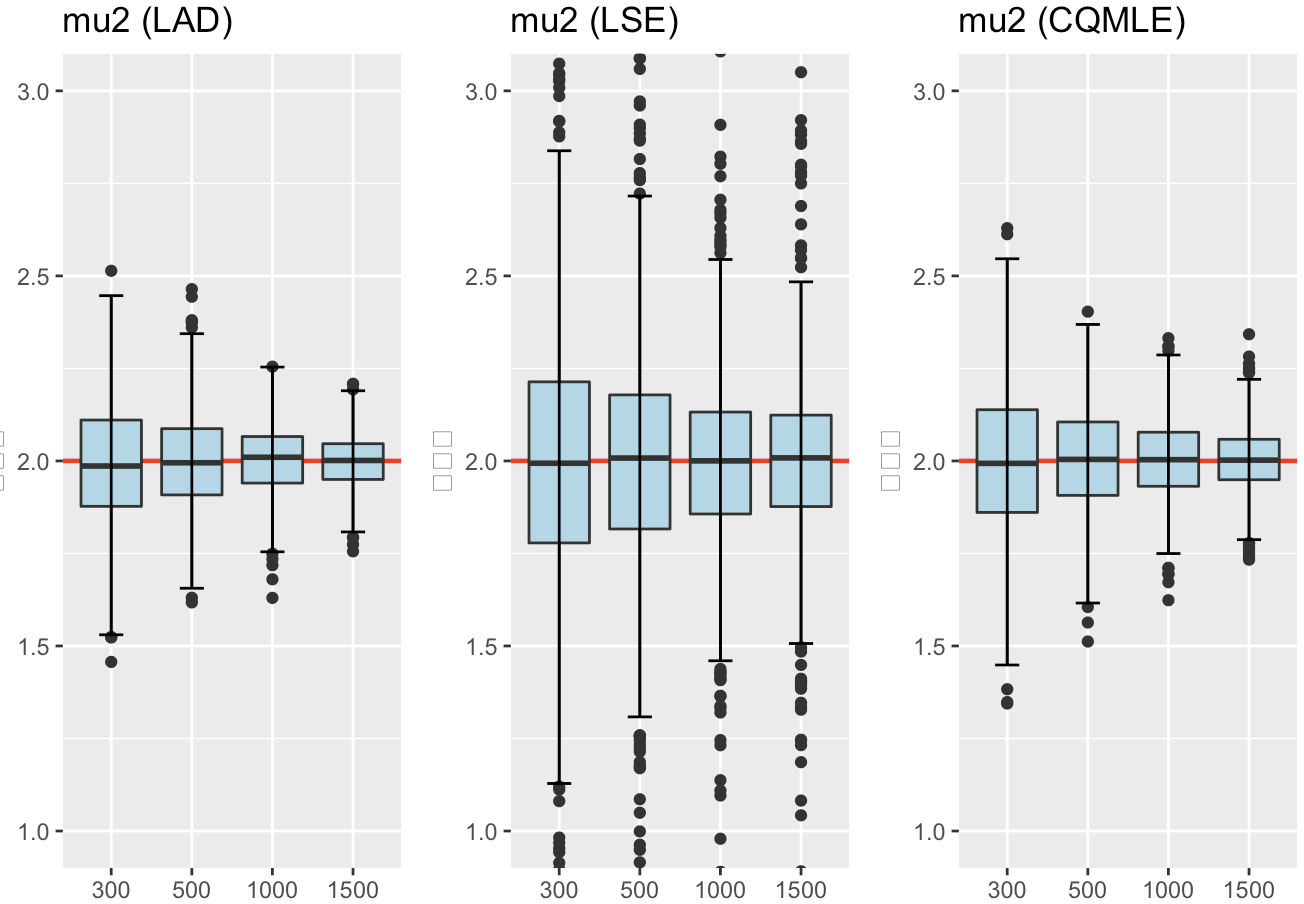}
  \end{minipage}
  \begin{minipage}[b]{0.45\linewidth}
    \centering
    \includegraphics[width=70mm, height=60mm]{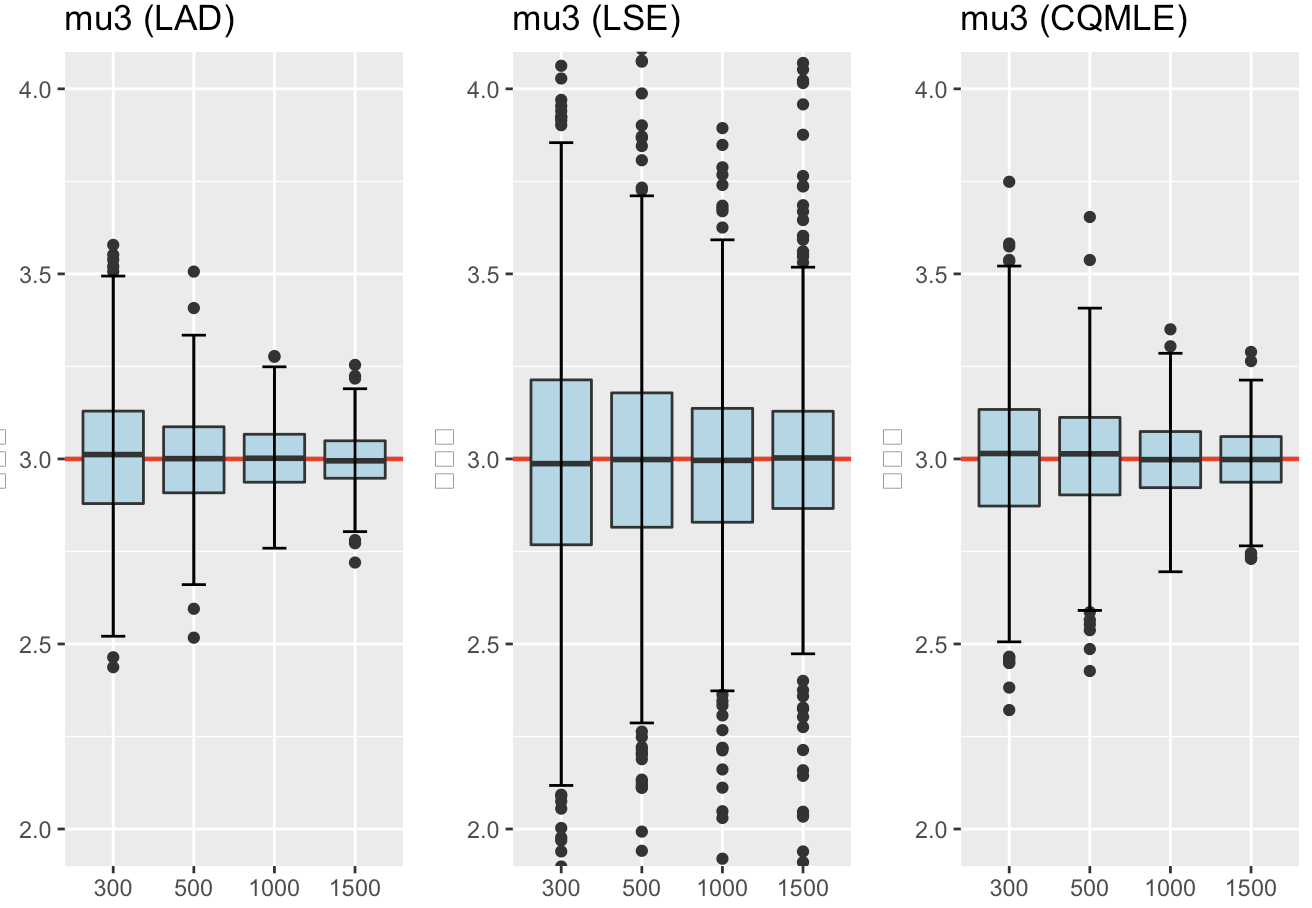}
  \end{minipage}
  \caption{Boxplots of the initial estimators for $n=300,~500,~1000,~1500$ when the true value is $(\al, \be, \sig, \mu)= (1.5, 0.5, 1.5, (5,2,3))$; in each panel, the red line shows the true value.
      }  
      \label{hm:fig_bp-3}
\end{figure}
\newpage
\begin{figure}[h]
  \begin{minipage}[b]{0.45\linewidth}
    \centering
    \includegraphics[width=60mm, height=60mm]{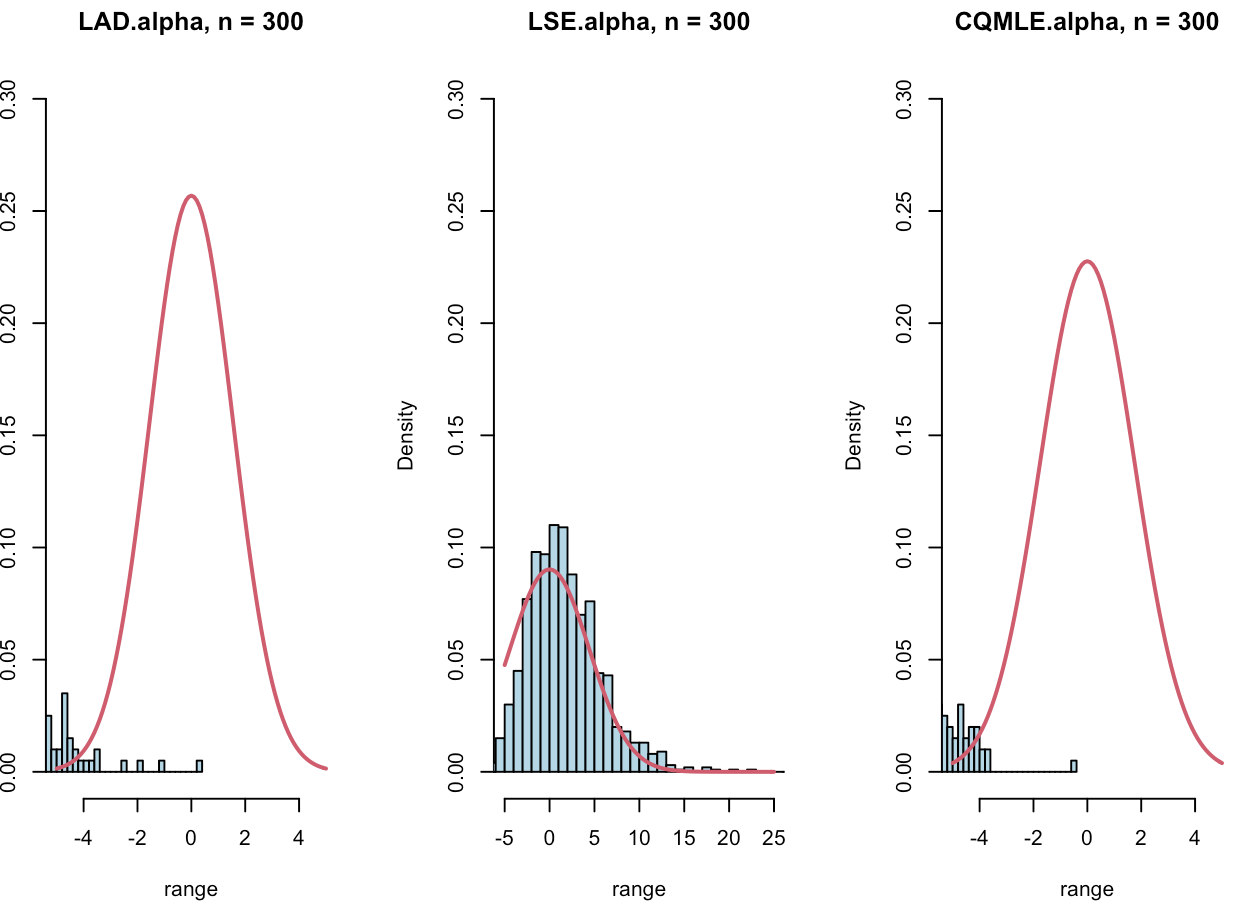}
  \end{minipage}
  \begin{minipage}[b]{0.45\linewidth}
    \centering
    \includegraphics[width=60mm, height=60mm]{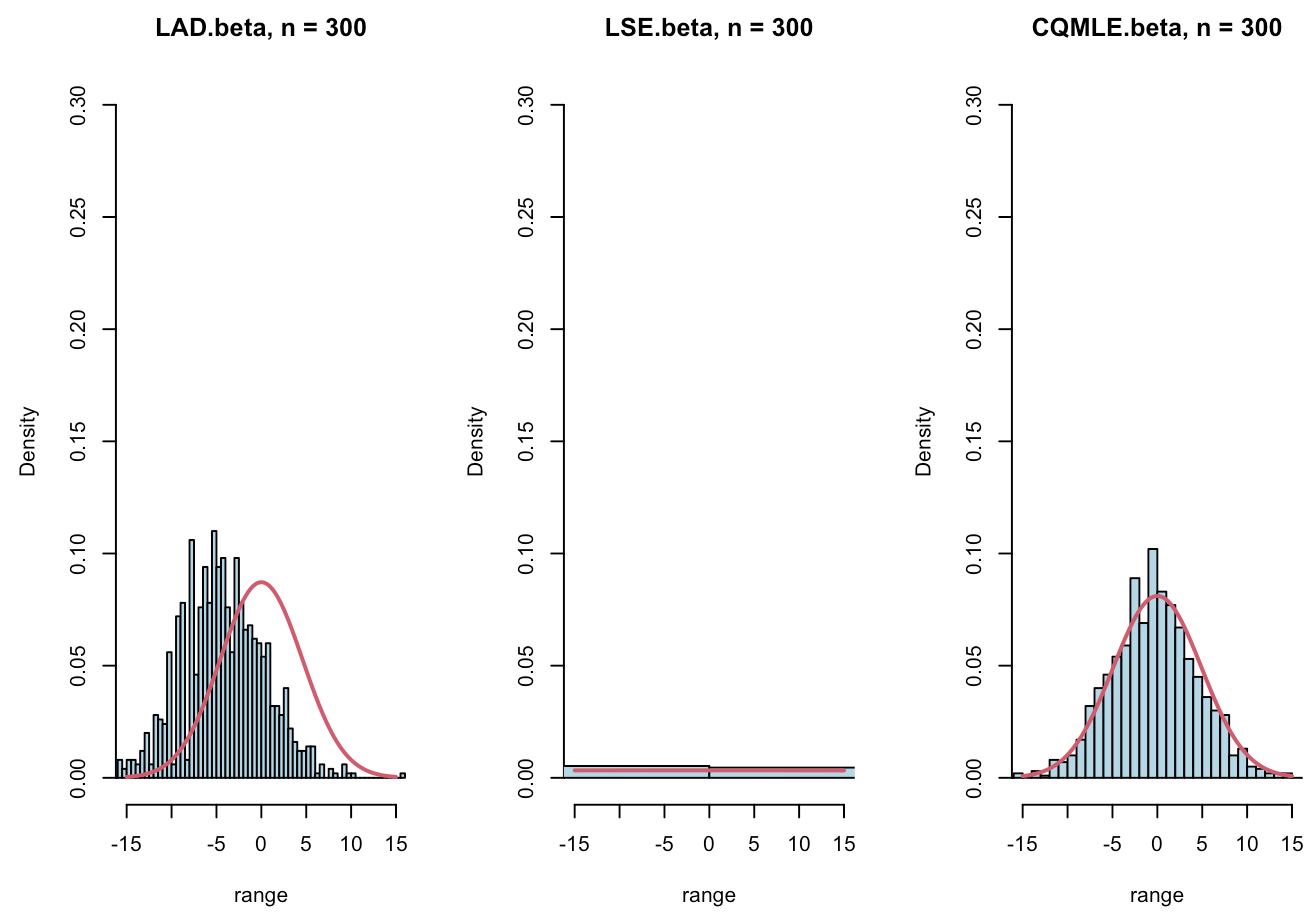}
  \end{minipage}
  \begin{minipage}[b]{0.45\linewidth}
    \centering
    \includegraphics[width=60mm, height=60mm]{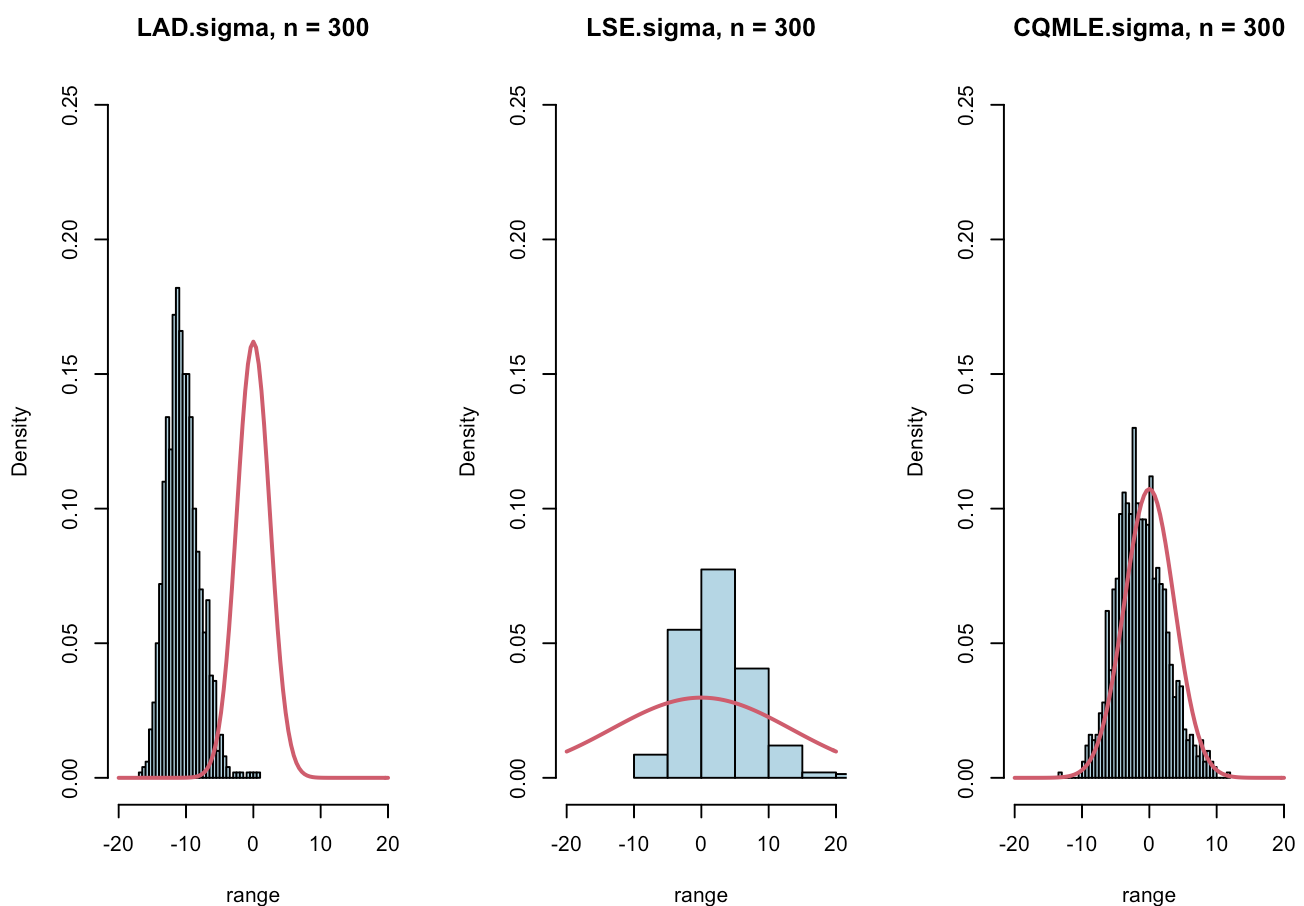}
  \end{minipage}
  \begin{minipage}[b]{0.45\linewidth}
    \centering
    \includegraphics[width=60mm, height=60mm]{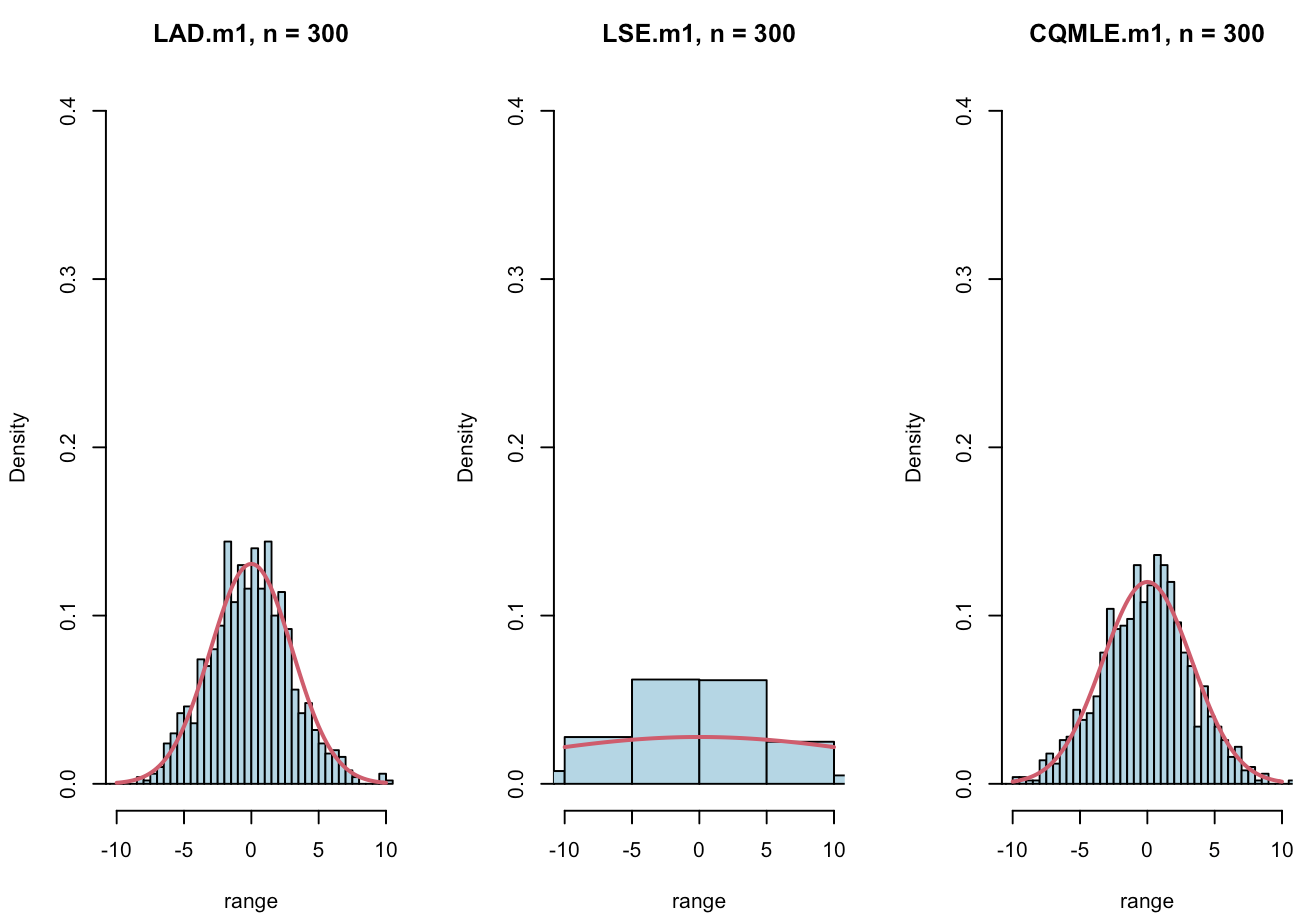}
  \end{minipage}
  \begin{minipage}[b]{0.45\linewidth}
    \centering
    \includegraphics[width=60mm, height=60mm]{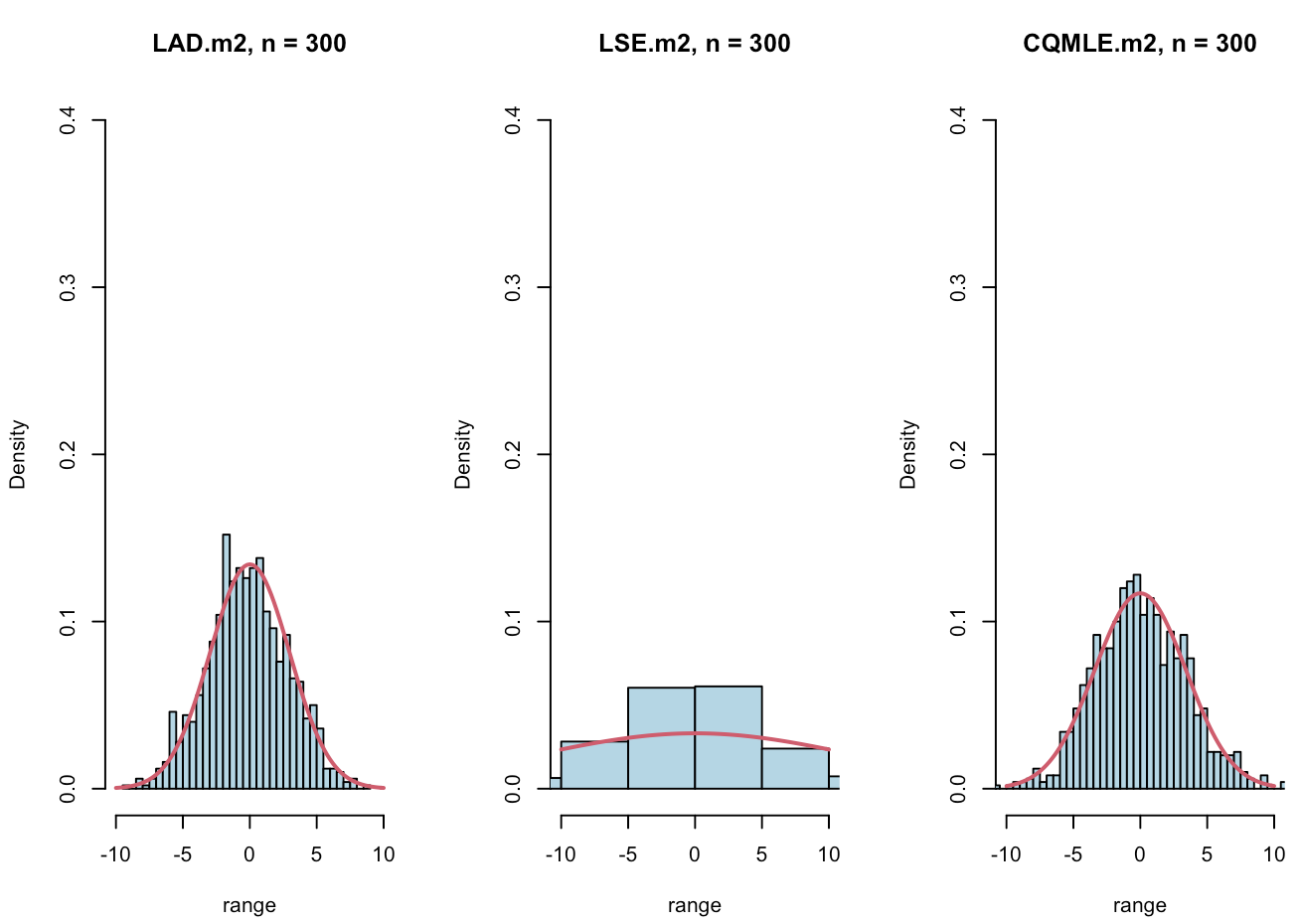}
  \end{minipage}
  \begin{minipage}[b]{0.45\linewidth}
    \centering
    \includegraphics[width=60mm, height=60mm]{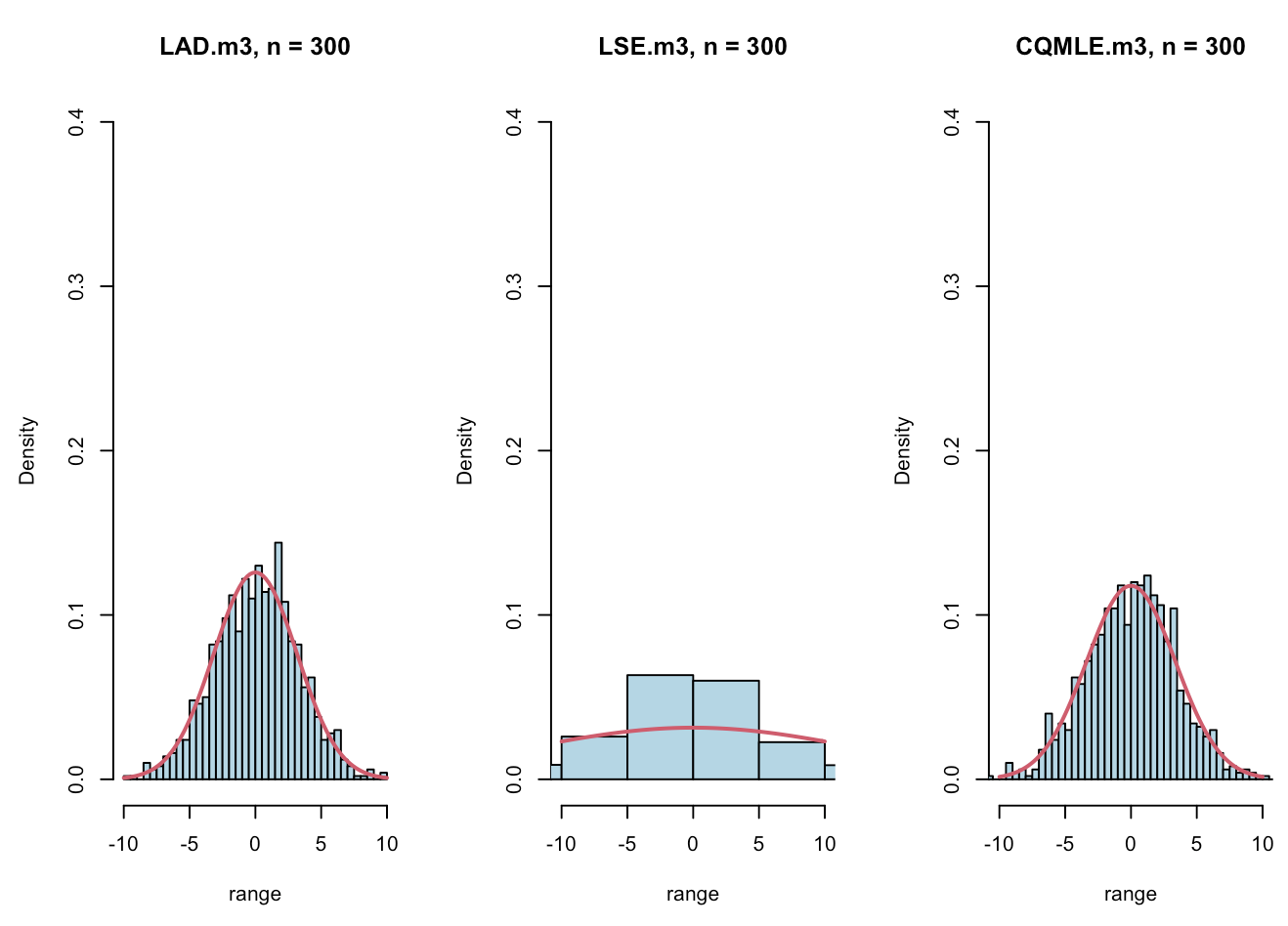}
  \end{minipage}
  \caption{Histograms of the initial estimators for $n=300$ when the true value is $(\al, \be, \sig, \mu)= (1.5, 0.5, 1.5, (5,2,3))$; in each panel, the red line represents the probability density function of the normal distribution with the mean 0 and the variance equal to that of the normalized estimator.
      }  
      \label{hm:fig_hist-15-1}
\end{figure}
\newpage
\begin{figure}[h]
  \begin{minipage}[b]{0.45\linewidth}
    \centering
    \includegraphics[width=60mm, height=60mm]{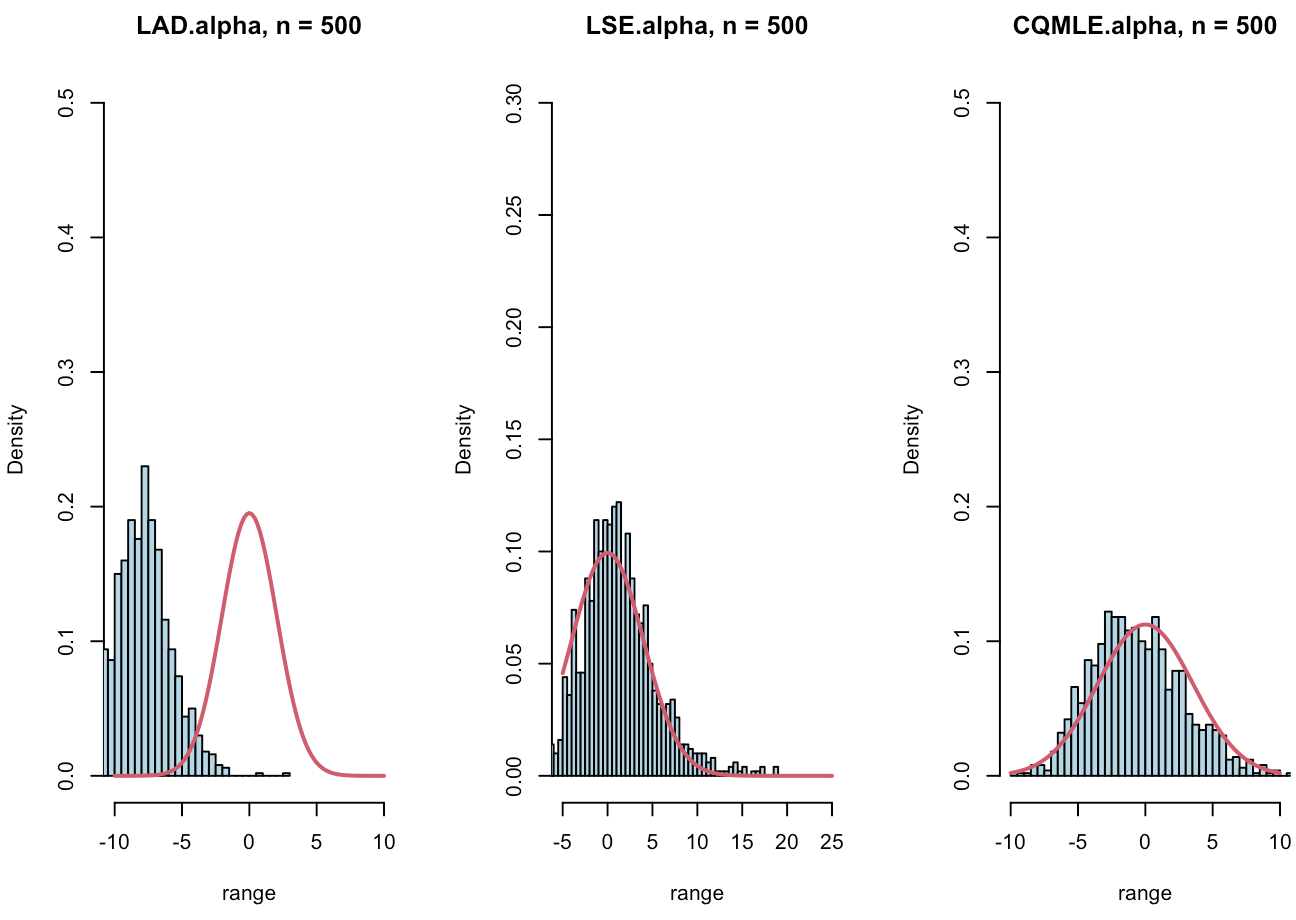}
  \end{minipage}
  \begin{minipage}[b]{0.45\linewidth}
    \centering
    \includegraphics[width=60mm, height=60mm]{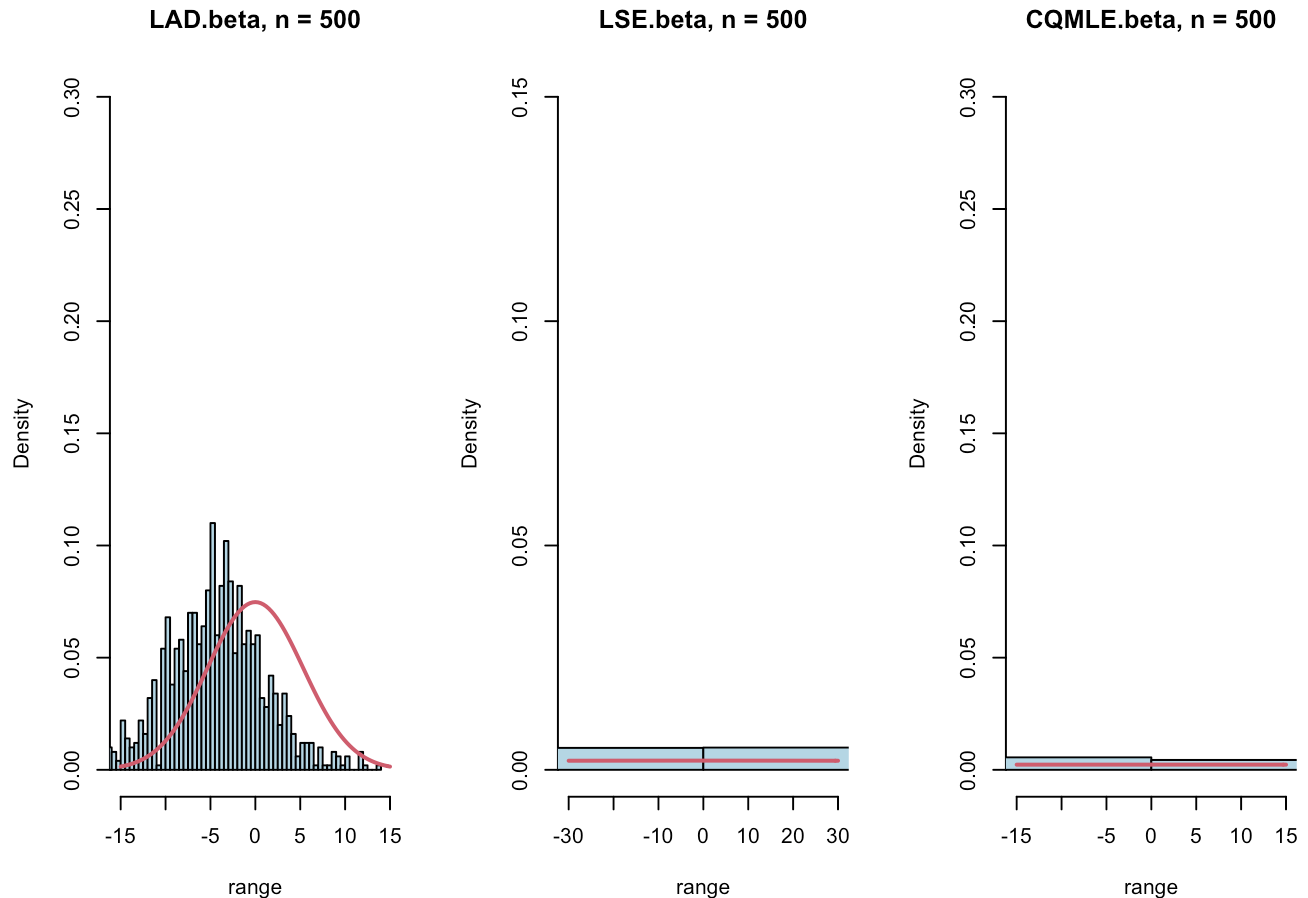}
  \end{minipage}
  \begin{minipage}[b]{0.45\linewidth}
    \centering
    \includegraphics[width=60mm, height=60mm]{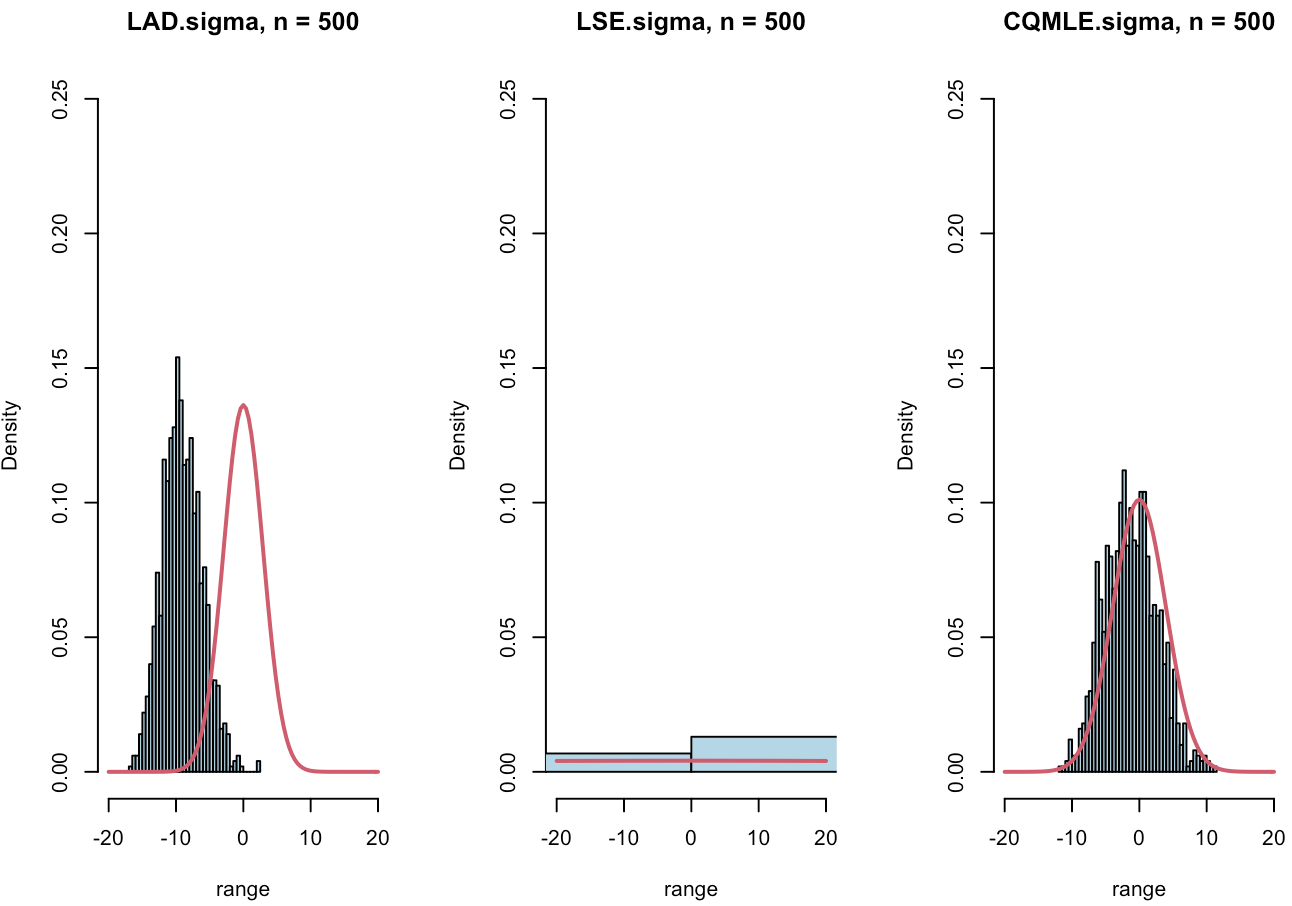}
  \end{minipage}
  \begin{minipage}[b]{0.45\linewidth}
    \centering
    \includegraphics[width=60mm, height=60mm]{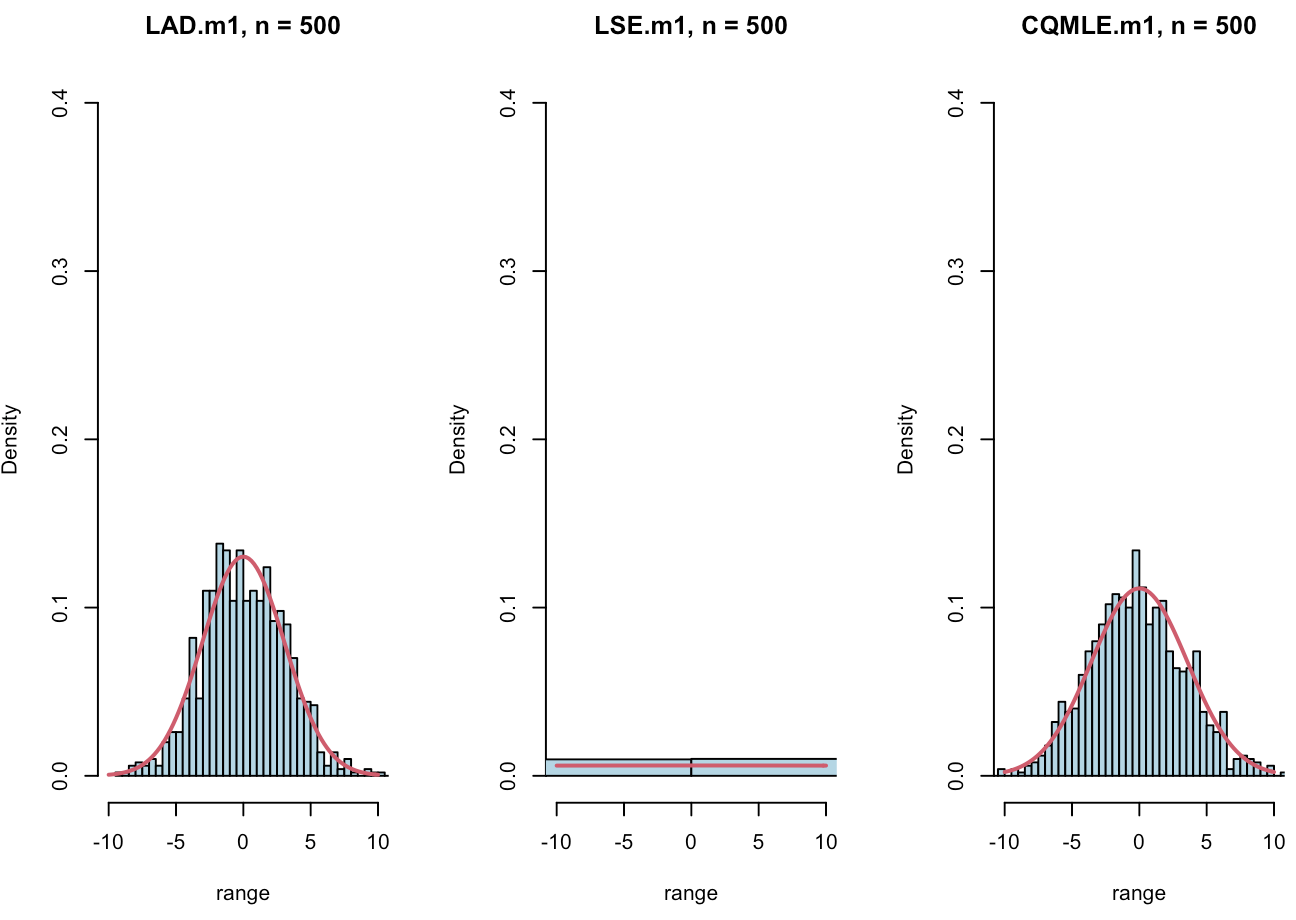}
  \end{minipage}
  \begin{minipage}[b]{0.45\linewidth}
    \centering
    \includegraphics[width=60mm, height=60mm]{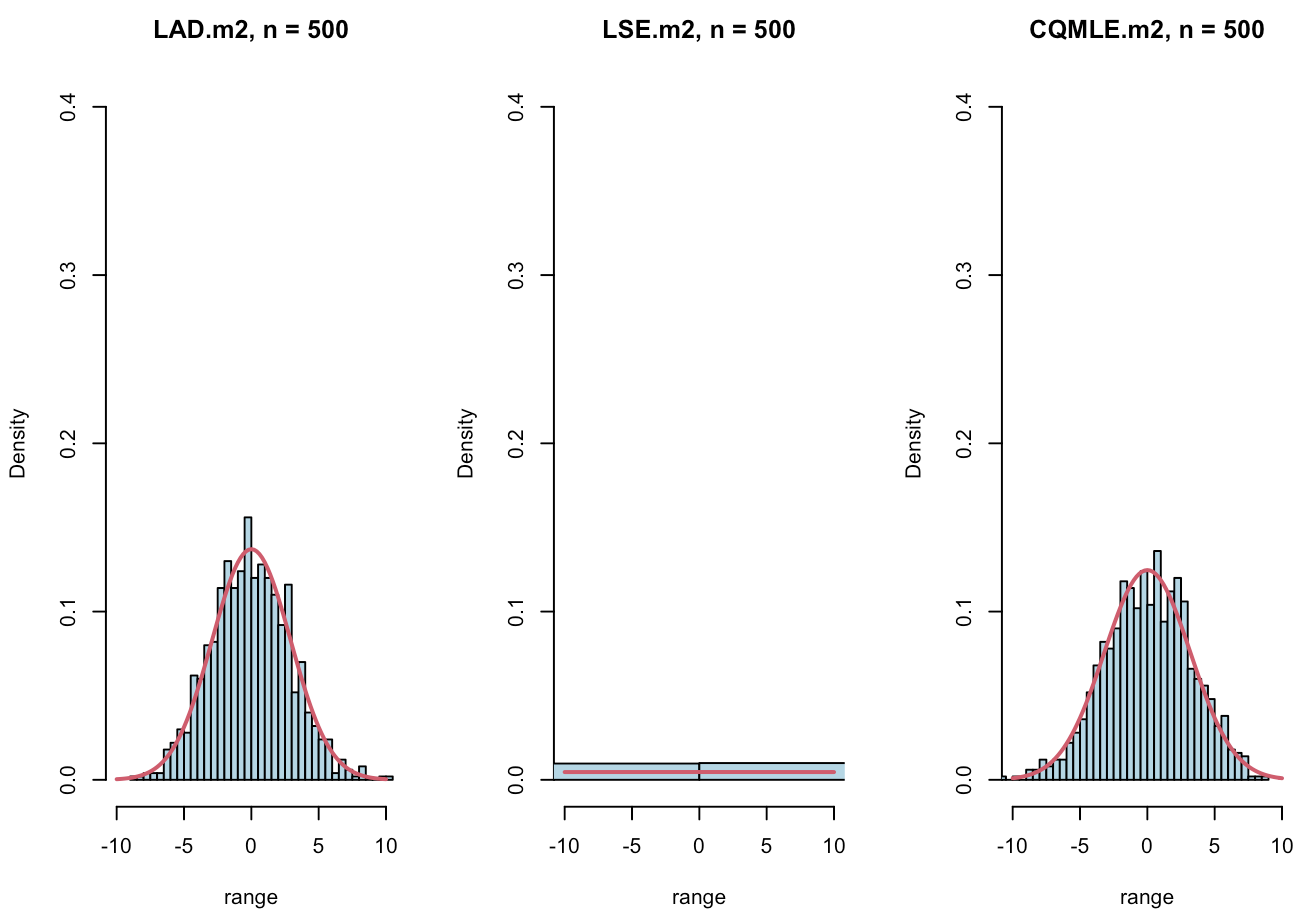}
  \end{minipage}
  \begin{minipage}[b]{0.45\linewidth}
    \centering
    \includegraphics[width=60mm, height=60mm]{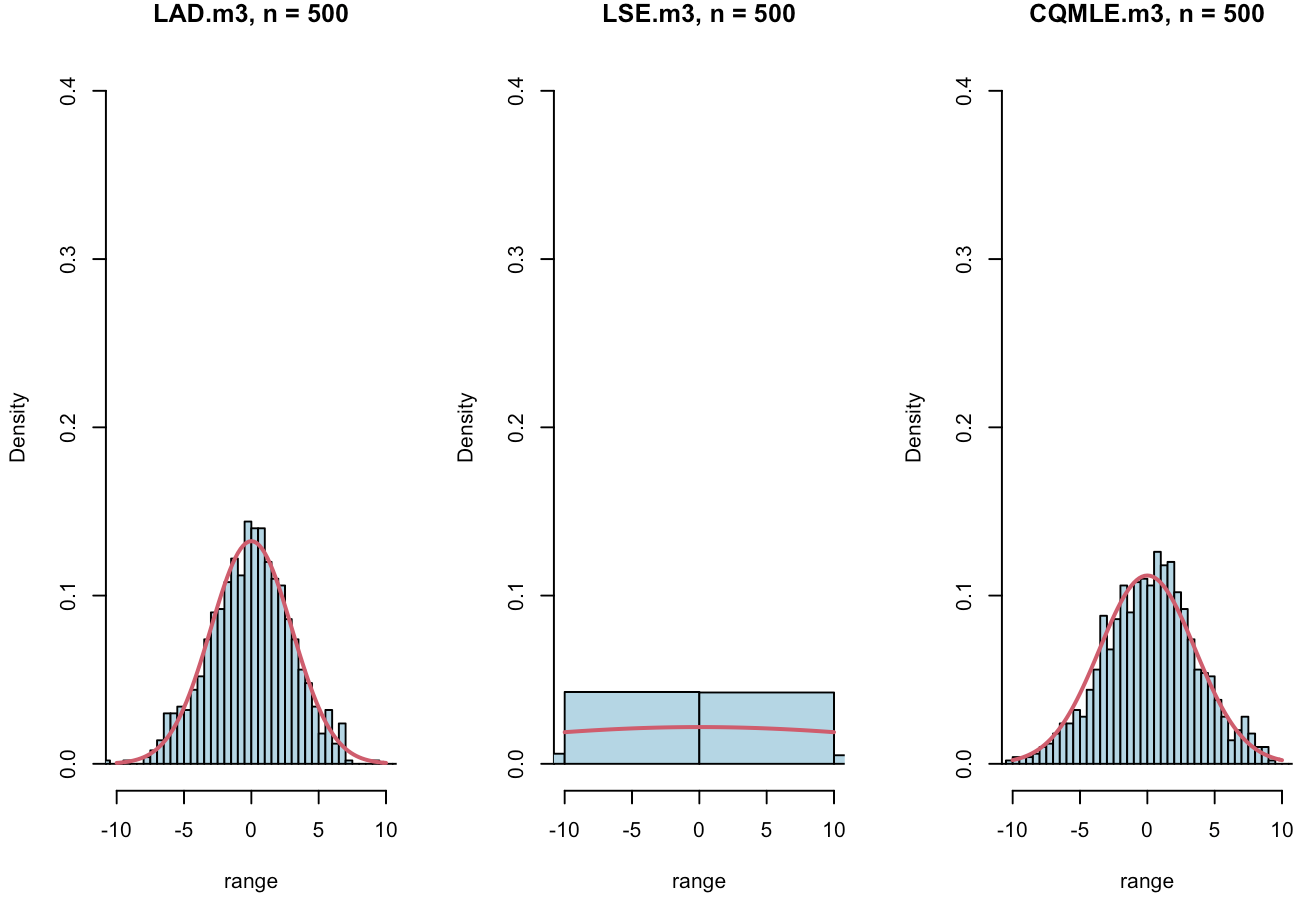}
  \end{minipage}
  \caption{Histograms of the initial estimators for $n=500$ when the true value is $(\al, \be, \sig, \mu)= (1.5, 0.5, 1.5, (5,2,3))$; in each panel, the red line represents the probability density function of the normal distribution with the mean 0 and the variance equal to that of the normalized estimator.
      }  
  \label{hm:fig_hist-15-2}
\end{figure}

\subsection{MLE with CQMLE-based initial value}\label{hm:sec_sim.MLE}

\begin{figure}[h]
  \begin{minipage}[b]{0.85\linewidth}
    \centering
    \includegraphics[
    scale=0.5
    ]{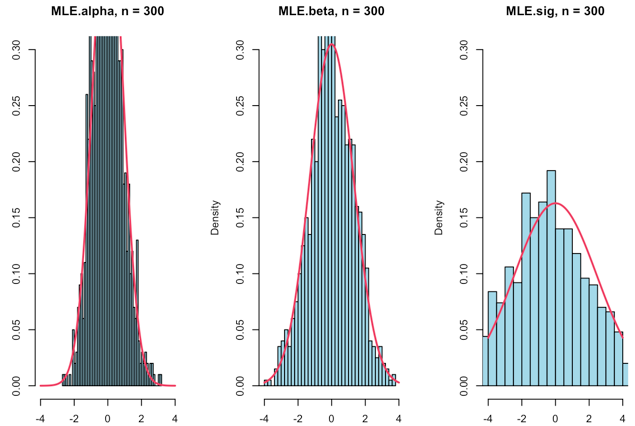}
  \end{minipage}
  \begin{minipage}[b]{0.85\linewidth}
    \centering
    \includegraphics[
    scale=0.5
    ]{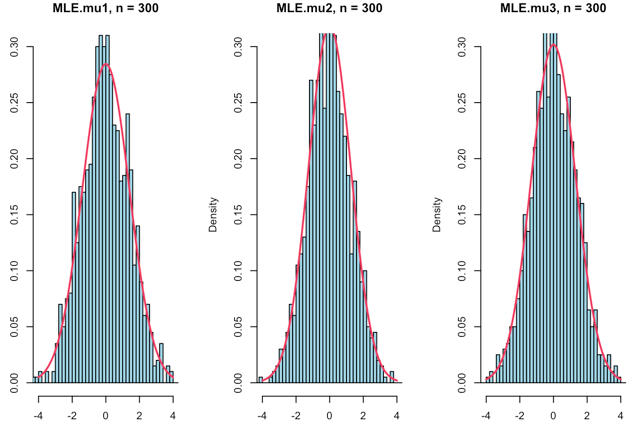}
  \end{minipage}
\caption{Histograms of the MLE for $n=300$ when we use the true value $(\al,\be, \sig, \mu)=(0.8,0.5,1.0,(5,2,3))$. \rev{
For the starting value for the numerical optimization of the log-likelihood, we used the CQMLE of $\mu$ and the moment estimator of $(\al,\be,\sig)$
}.}
\label{hm:fig_mle-1}
\end{figure}
\begin{figure}[h]
\begin{minipage}[b]{0.85\linewidth}
    \centering
    \includegraphics[
    scale=0.5
    ]{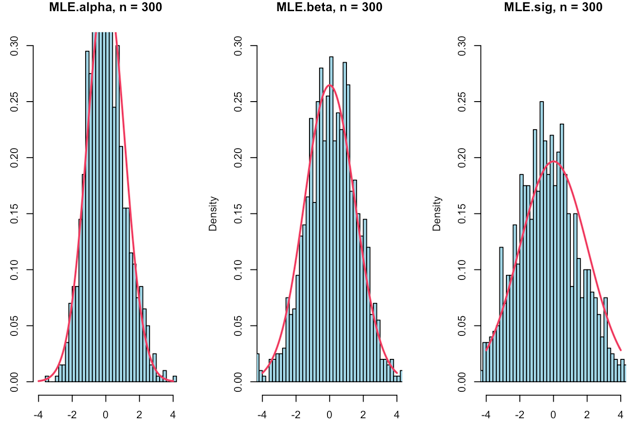}
  \end{minipage}
  \begin{minipage}[b]{0.85\linewidth}
    \centering
    \includegraphics[
    scale=0.5
    ]{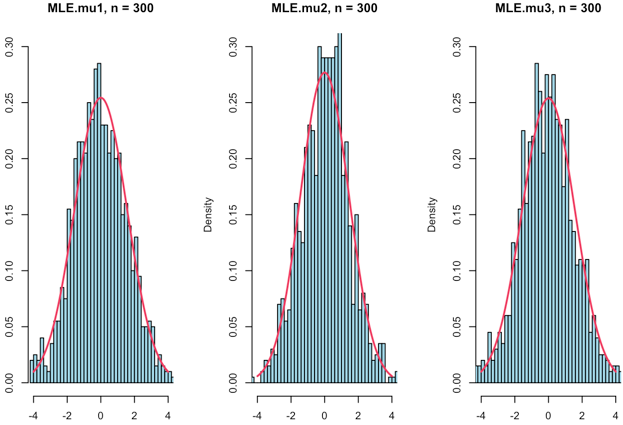}
  \end{minipage}
\caption{Histograms of the MLE for $n=300$ when we use the true value $(\al,\be, \sig, \mu)=(1.0,0.5,1.0,(5,2,3))$. \rev{
For the starting value for the numerical optimization of the log-likelihood, we used the CQMLE of $\mu$ and the moment estimator of $(\al,\be,\sig)$
}.}
\label{hm:fig_mle-2}
\end{figure}

\newpage

\begin{figure}[h]
\begin{minipage}[b]{0.85\linewidth}
    \centering
    \includegraphics[
    scale=0.5
    ]{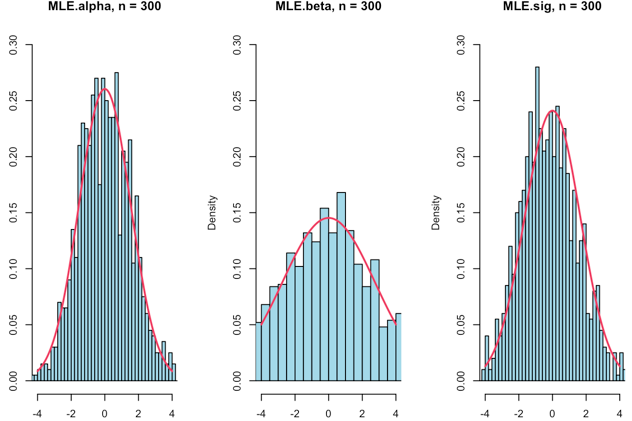}
  \end{minipage}
  \begin{minipage}[b]{0.85\linewidth}
    \centering
    \includegraphics[
    scale=0.5
    ]{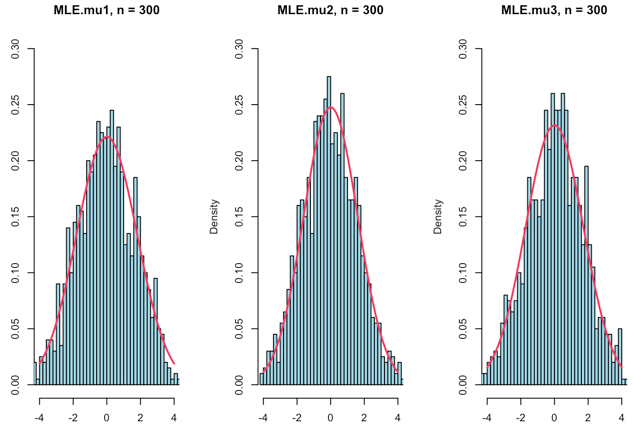}
  \end{minipage}
\caption{Histograms of the MLE for $n=300$ when we use the true value $(\al,\be, \sig, \mu)=(1.5,0.5,1.0,(5,2,3))$. \rev{
For the starting value for the numerical optimization of the log-likelihood, we used the CQMLE of $\mu$ and the moment estimator of $(\al,\be,\sig)$
}.}
\label{hm:fig_mle-3}
\end{figure}

\tcr{
We show the results in Figures \ref{hm:fig_mle-1},\ref{hm:fig_mle-2}, and \ref{hm:fig_mle-3}.
}
Histograms show that the MLE with the CQMLE as an initial estimate for numerical search has good performance for each true value, in particular free from the local optimum problem.

\tcr{
We additionally conducted simulations under the parameter settings $\alpha =0.8, 1.0, 1.5$, $\beta = -0.5$, $\sigma = 1.5$, and $\mu = (5, 2, 3)$ with sample sizes $n = 300$ and $n = 500$, and computed the MLEs with CQMLE-based initial value. 
Since the results under these parameter settings were very similar to the previous ones, we show them in Tables \ref{tab:0.8_mean_estimates_sd}, \ref{tab:1.0_mean_estimates_sd}, and \ref{tab:1.5_mean_estimates_sd} without histograms and boxplots.
}

\begin{table}[ht]
\centering
\begin{tabular}{c|c|c|c|c|c|c}
  \hline
$n$ & $\aes$ & $\bes$ & $\ses$ & $\hat{\mu}_1$ & $\hat{\mu}_2$ & $\hat{\mu}_3$ \\
  \hline
  300 & 0.7970 & -0.4988 & 1.4864 & 4.9988 & 2.0032 & 2.9997 \\
 & (0.0480) & (0.0775) & (0.1394) & (0.0793) & (0.0716) & (0.0716) \\
500 & 0.7991 & -0.5008 & 1.4925 & 5.0016 & 1.9996 & 2.9982 \\
 & (0.0381) & (0.0579) & (0.1059) & (0.0556) & (0.0581) & (0.0581) \\
  \hline
\end{tabular} 
\caption{Means and standard deviations of the estimates for the true values: $\alpha = 0.8$, $\beta = -0.5$, $\sigma = 1.5$, $\mu = (5, 2, 3)$.}
\label{tab:0.8_mean_estimates_sd}
\end{table}


\begin{table}[ht]
\centering
\begin{tabular}{c|c|c|c|c|c|c}
  \hline
  $n$ & $\hat{\alpha}$ & $\hat{\beta}$ & $\hat{\sigma}$ & $\hat{\mu}_1$ & $\hat{\mu}_2$ & $\hat{\mu}_3$ \\
  \hline
  300 & 0.99991 & -0.49973 & 1.4890 & 5.0022 & 1.997 & 3.0011 \\
   & (0.0651) & (0.0881) & (0.1186) & (0.0933) & (0.0855) & (0.0855) \\
  500 & 1.00094 & -0.50143 & 1.4899 & 5.0000 & 2.001 & 3.0000 \\
   & (0.0496) & (0.0713) & (0.0903) & (0.0660) & (0.0672) & (0.0672) \\
  \hline
\end{tabular}
\caption{Means and standard deviations of the estimates for the true values: $\alpha = 1.5$, $\beta = -0.5$, $\sigma = 1.5$, $\mu = (5, 2, 3)$.}
\label{tab:1.0_mean_estimates_sd}
\end{table}
\begin{table}[ht]
\centering
\begin{tabular}{c|c|c|c|c|c|c}
  \hline
  $n$ & $\aes$ & $\bes$ & $\ses$ & $\hat{\mu}_1$ & $\hat{\mu}_2$ & $\hat{\mu}_3$ \\
  \hline
  300 & 1.5022 & -0.5030 & 1.4911 & 5.0013 & 2.0032 & 2.9936 \\
   & (0.0893) & (0.1565) & (0.0959) & (0.1018) & (0.0942) & (0.0942) \\
  500 & 1.4995 & -0.5036 & 1.4909 & 4.9967 & 2.0009 & 3.0038 \\
   & (0.0664) & (0.1218) & (0.0724) & (0.0712) & (0.0739) & (0.0739) \\
  \hline
\end{tabular}
\caption{Means and standard deviations of the estimates for the true values: $\alpha = 1.5$, $\beta = -0.5$, $\sigma = 1.5$, $\mu = (5, 2, 3)$.}
\label{tab:1.5_mean_estimates_sd}
\end{table}

\vspace{1em}




\bigskip

\noindent
\textbf{Acknowledgements.}
The authors are grateful to the reviewers for their valuable comments.
The first author (EK) thanks JGMI of Kyushu University for their support. This work was partially supported by JST CREST Grant Number JPMJCR2115 and JSPS KAKENHI Grant Number 22H01139, Japan (HM).

\bigskip

\bibliographystyle{abbrv} 

\begin{thebibliography}{10}

\bibitem{Ada73}
R.~A. Adams.
\newblock Some integral inequalities with applications to the imbedding of
  sobolev spaces defined over irregular domains.
\newblock {\em Transactions of the American Mathematical Society},
  178(0):401--429, 1973.

\bibitem{AndCalDav09}
B.~Andrews, M.~Calder, and R.~A. Davis.
\newblock Maximum likelihood estimation for {$\alpha$}-stable autoregressive
  processes.
\newblock {\em Ann. Statist.}, 37(4):1946--1982, 2009.

\bibitem{CalDav98}
M.~Calder and R.~Davis.
\newblock Inference for linear processes with stable noise.
\newblock In {\em A practical guide to heavy tails}, pages 159--176. Birkhauser
  Boston Inc., 1998.

\bibitem{FanQiXiu14}
J.~Fan, L.~Qi, and D.~Xiu.
\newblock Quasi-maximum likelihood estimation of {GARCH} models with
  heavy-tailed likelihoods.
\newblock {\em J. Bus. Econom. Statist.}, 32(2):178--191, 2014.

\bibitem{Fer96}
T.~S. Ferguson.
\newblock {\em A course in large sample theory}.
\newblock Texts in Statistical Science Series. Chapman \& Hall, London, 1996.

\bibitem{Hos22_mt}
Y.~Hosokawa.
\newblock Estimation of infinite-variance linear regression model.
\newblock Master's thesis, Kyushu University, 2022.

\bibitem{Kni98}
K.~Knight.
\newblock Limiting distributions for $l_1$ regression estimators under general
  conditions.
\newblock {\em Annals of Statistics}, 26:755--770, 1998.

\bibitem{Kou80-2}
I.~A. Koutrouvelis.
\newblock Regression-type estimation of the parameters of stable laws.
\newblock {\em J. Amer. Statist. Assoc.}, 75(372):918--928, 1980.

\bibitem{Kur01}
E.~E. Kuruo{\u{g}}lu.
\newblock Density parameter estimation of skewed {$\alpha$}-stable
  distributions.
\newblock {\em IEEE Trans. Signal Process.}, 49(10):2192--2201, 2001.

\bibitem{Mat21}
M.~Matsui.
\newblock Asymptotics of maximum likelihood estimation for stable law with
  continuous parameterization.
\newblock {\em Comm. Statist. Theory Methods}, 50(15):3695--3712, 2021.

\bibitem{Mut17}
R.~Mutwiri, H.~Mwambi, and P.~Slotow.
\newblock An application of stable Paretian time series models to animal
  movement GPS telemetry data.
\newblock pages 2308--1365, 01 2019.

\bibitem{NikSha95}
C.~L. Nikias and M.~Shao.
\newblock {\em Signal processing with alpha-stable distributions and
  applications}.
\newblock Wiley-Interscience, 1995.

\bibitem{Nol98}
J.~P. Nolan.
\newblock Parameterizations and modes of stable distributions.
\newblock {\em Statist. Probab. Lett.}, 38(2):187--195, 1998.

\bibitem{Nol20}
J.~P. Nolan.
\newblock {\em Univariate stable distributions}.
\newblock Springer Series in Operations Research and Financial Engineering.
  Springer, Cham, [2020] \copyright 2020.
\newblock Models for heavy-tailed data.

\bibitem{Nol13}
J.~P. Nolan and D.~Ojeda-Revah.
\newblock Linear and nonlinear regression with stable errors.
\newblock {\em Journal of Econometrics}, 172(2):186--194, 2013.
\newblock Latest Developments on Heavy-Tailed Distributions.

\bibitem{Pre72}
S.~J. Press.
\newblock Estimation in univariate and multivariate stable distributions.
\newblock {\em Journal of the American Statistical Association},
  67(340):842--846, 1972.

\bibitem{Agu.et.al19}
R.~Rodriguez-Aguilar, J.~A. Marmolejo-Saucedo, and B.~Retana-Blanco.
\newblock {Prices of Mexican Wholesale Electricity Market: An Application of
  Alpha-Stable Regression}.
\newblock {\em Sustainability}, 11(11):1--14, June 2019.

\bibitem{Sat99}
K.-i. Sato.
\newblock {\em L\'evy processes and infinitely divisible distributions},
  volume~68 of {\em Cambridge Studies in Advanced Mathematics}.
\newblock Cambridge University Press, Cambridge, 1999.
\newblock Translated from the 1990 Japanese original, Revised by the author.

\bibitem{Sha69}
M.~Sharpe.
\newblock Zeroes of infinitely divisible densities.
\newblock {\em Ann. Math. Statist.}, 40:1503--1505, 1969.

\bibitem{vdV98}
A.~W. van~der Vaart.
\newblock {\em Asymptotic statistics}, volume~3 of {\em Cambridge Series in
  Statistical and Probabilistic Mathematics}.
\newblock Cambridge University Press, Cambridge, 1998.

\bibitem{Yos11}
N.~Yoshida.
\newblock Polynomial type large deviation inequalities and quasi-likelihood
  analysis for stochastic differential equations.
\newblock {\em Ann. Inst. Statist. Math.}, 63(3):431--479, 2011.

\bibitem{Zac71}
S.~Zacks.
\newblock {\em The theory of statistical inference}.
\newblock John Wiley \& Sons, Inc., New York-London-Sydney, 1971.
\newblock Wiley Series in Probability and Mathematical Statistics.

\end{thebibliography}

\end{document}